\documentclass[11pt]{article}
\usepackage{amsfonts}
\usepackage{amssymb}
\usepackage{textcomp}
\usepackage{graphicx}
\usepackage{pdfsync}
\usepackage{color}
\usepackage{pbsi}
\def\sqr#1#2{{\vcenter{\vbox{\hrule height.#2pt
        \hbox{\vrule width.#2pt height#1pt \kern#2pt
        \vrule width.#2pt}
        \hrule height.#2pt}}}}

\newcommand{\nc}{\newcommand}
\nc{\parent}[1]{$[\![#1]\!]$}
\newtheorem{theorem}{Theorem}[section]
\newtheorem{lemma}{Lemma}[section]
\newtheorem{example}{Example}[section]
\newtheorem{corollary}{Corollary}[section]
\newtheorem{proposition}{Proposition}[section]
\newtheorem{remark}{Remark}[section]
\newtheorem{definition}{Definition}[section]
\newtheorem{assumption}{Assumption}[section]

\newenvironment{proof}{{\sc Proof.}\hspace{3mm}}{\qed}
\newenvironment{pf-main}{{\sc Proof of Theorem \ref{mainresult}.}\hspace{3mm}}{\qed}
\DeclareMathAlphabet{\mscr}{T1}{pbsi}{m}{it}
\DeclareMathAlphabet{\mathpzc}{OT1}{pzc}{m}{it}
\nc{\cadlag}{c\`{a}dl\`{a}g } \nc{\ba}{\begin{array}}
\nc{\ea}{\end{array}} \nc{\be}{\begin{equation}}
\nc{\ee}{\end{equation}} \nc{\bea}{\begin{eqnarray}}
\nc{\eea}{\end{eqnarray}} \nc{\bean}{\begin{eqnarray*}}
\nc{\eean}{\end{eqnarray*}} \nc{\bu}{\bullet} \nc{\nn}{\nonumber}
\nc{\cA}{{\mathcal A}} \nc{\cB}{{\mathcal B}} \nc{\cC}{{\mathcal
C}} \nc{\cD}{{\mathcal D}} \nc{\bbD}{\mathbb{D}}
\nc{\cG}{{\mathcal G}} \nc{\cF}{{\mathcal F}} \nc{\cS}{{\mathcal
S}} \nc{\cU}{{\mathcal U}} \nc{\cH}{{\mathcal H}}
\nc{\cK}{{\mathcal K}}\nc{\cL}{{\mathcal L}}  \nc{\cM}{{\mathcal
M}} \nc{\cO}{{\mathcal O}} \nc{\cP}{{\mathcal P}}
\nc{\bbE}{\mathbb{E}} \nc{\bbF}{\mathbb{F}}
\nc{\bbEQ}{\mathbb{E}_{\mathbb{Q}}} \nc{\eps}{\varepsilon}
\nc{\bbEP}{\mathbb{E}_{\mathbb{P}}}\nc{\bbL}{\mathbb{L}}
\nc{\what}{\widehat} \nc{\bbP}{\mathbb{P}} \nc{\bbQ}{\mathbb{Q}}
\nc{\del}{\partial} \nc{\Om}{\Omega} \nc{\om}{\omega}
\nc{\bbR}{\mathbb{R}} \nc{\bbN}{\mathbb{N}} \nc{\fps}{$(\Om, \cF,
(\cF_t)_{t\geq 0}, \bbP)$} \nc{\bbC}{\mathbb{C}}
\nc{\bfr}{\begin{flushright}} \nc{\efr}{\end{flushright}}
\nc{\dXt}{\Delta X_{t}} \nc{\dXs}{\Delta X_{s}}
\nc{\bs}{\blacksquare} \nc{\dX}{\Delta X} \nc{\dY}{\Delta Y}
\nc{\dnkx}{\left(X(T^{n}_{k})-X(T^{n}_{k-1})\right)}
\nc{\esssup}{\mathrm{ess}\mbox{ }\mathrm{sup}}
\nc{\essinf}{\mathrm{ess}\mbox{ } \mathrm{inf}}
\nc{\dhats}{\widehat{\delta_s}} \nc{\half} {\frac{1}{2}}

\nc{\norm}{\parallel}
\nc{\qed}{\hfill $\blacksquare$}
\nc{\ol}{\overline}
\def\rar{\rightarrow}

\nc{\chf}{\mbox{$\mathbf1$}}
\begin{document}
\title{Dynamic Markov bridges motivated by models of insider trading\footnote{This research benefited from the support of the ``Chair Les Particuliers Face aux Risques'', Fondation du Risque (Groupama-ENSAE-Dauphine), the GIP-ANR ``Croyances'' project as well as from the ``Chaire Risque de cr\'edit'', F\'ed\'eration Bancaire Fran\c caise, of the
Europlace Institute of Finance, for their support. Part of this research has been done when the second author was visiting Evry University during September 2008.}}
\author{Luciano Campi\footnote{CEREMADE, University Paris-Dauphine, campi@ceremade.dauphine.fr.}\qquad Umut \c Cetin\footnote{Department of Statistics, London School of Economics, u.cetin@lse.ac.uk.} \qquad Albina Danilova\footnote{Department of Mathematics, London School of Economics, a.danilova@lse.ac.uk.}}
\maketitle

\begin{abstract} Given a Markovian Brownian martingale $Z$, we build a
process $X$ which is a martingale in its own filtration and
satisfies $X_1 = Z_1$. We call $X$ a dynamic bridge, because its terminal value $Z_1$ is not known in advance.
We compute explicitly its semimartingale
decomposition under both its own filtration $\cF^X$ and the filtration $\cF^{X,Z}$ jointly generated by $X$ and $Z$. Our construction is heavily based on parabolic PDE's and filtering techniques.
As an application, we explicitly
solve an equilibrium model with insider trading, that can be viewed
as a non-Gaussian generalization of Back and Pedersen's \cite{BP}, where insider's additional information evolves over time.\\


\textbf{AMS classification (2000):} 60G44, 60H05, 60H10, 93E11

\end{abstract}

\section{Introduction}
Consider two independent Brownian motions $B$ and $\beta$ over the time interval $[0,1]$ and define a signal process $Z$ as the unique strong solution to
\[dZ_t = \sigma (t)a(V(t),Z_t)d\beta_t ,\]
where $\sigma:[0,1]\mapsto \bbR_+$ is a deterministic function, $V(t):=c+\int_0^t\sigma^2(s)\, ds$ for some constant $c>0$, and $a:[0,1]\times \bbR\mapsto \bbR$ is regular enough for ensuring the existence of a unique strong solution. Moreover, $\sigma, V$ and $a$ are required to satisfy further regularity conditions precise statements of which are given in Assumptions \ref{AssZ},  \ref{fs} and \ref{bt}.

We are interested in the construction of a Markov process $X$ which is a martingale in its own filtration and such that $X_1 = Z_1$. This construction will be performed adding a well-chosen drift to a suitable Brownian martingale corresponding to $B$. Such a drift will be a nonlinear function of $Z_t$ and $X_t$. Our goal is to obtain the Doob-Meyer decomposition of $X$ under both filtrations $\cF^X$ and $\cF^{X,Z}$, i.e. that generated by $X$ itself and that generated jointly by $X$ and the signal $Z$. We called such a process a dynamic Markov bridge because it is Markov and especially because its terminal value is not fixed in advance but it is dynamic itself, being the terminal value of the process $Z$. This construction is obtained in Theorem \ref{mainresult}, which is the main result of the mathematical part of the paper.

This study has a two-fold motivation, probabilistic and financial. First, the purely probabilistic one: in the paper by F\"ollmer et al. \cite{fwy}, a thorough investigation of this problem has been done in the case $a\equiv \sigma \equiv 1$, where $Z$ as well as $X$ are Gaussian processes. More precisely, they studied solutions $X$ of SDE's $dX_t = dB_t + \alpha_t dt$, where $\alpha_t$ depends \emph{linearly} on $X$ and $Z=\beta$. They obtain a characterization of such linear drifts $\alpha_t$ making $X$ a Brownian motion in its own filtration in terms of Volterra kernels solutions to some integral equations, that can be reduced in some special case to a Sturm-Liouville equation. F\"ollmer et al. were in turn motivated by the following natural modification of the classical Brownian bridge dynamics
\[dX_t = dB_t + \frac{\beta_t - X_t}{1-t}dt\]
where in the drift $\beta_t$ replaces $\beta_1$ as in the classical ``static'' Brownian bridge. It has been shown in \cite{fwy} that $X$ is still a bridge in the sense that $X_1 =\beta_1$ but it is not a Brownian motion in its own filtration anymore. They then focused more on general linear drifts preserving the Brownian property. Considering only  linear drifts allows them to use the nice and powerful relation between Gaussian processes and Volterra kernels. Few questions naturally arise from that work: What happens if the signal $Z$ is not necessarily Gaussian? Is it still possible to construct a ``dynamic'' bridge $X$ with the required properties? It turns out that it is still possible, but using completely different techniques. Indeed, $Z$ being  not Gaussian anymore, one is lead to consider nonlinear drifts to build the bridge, which makes impossible the use of Gaussian processes theory. However, the Markov nature of the problem allows us to use techniques from parabolic PDEs and those from filtering theory in order to carry out our analysis.

The second motivation -- that we share with F\"ollmer et al. \cite{fwy}  -- is a financial one: The dynamic bridge $X$ is the solution of a Kyle-Back type equilibrium model of a gradually informed insider trading (see \cite{B} for initial information and \cite{BP, Wu, Danilova} for the dynamic information case). In such a model, the insider observes a signal process (unknown to the market) $Z$ as above with $a\equiv 1$ driven by the Brownian motion $\beta$. She applies a well-chosen drift, modelling her strategy, to the Brownian motion $B$ in such a way that (i) the resulting process $X$ ends up in $Z_1$ and (ii) the distribution of the process remains unchanged, i.e. $X$ is again a Brownian motion in its own filtration. Condition (i) guarantees that the strategy maximizes the insider's expected gain. On the other hand, condition (ii) means that the strategy of the insider, i.e. the drift, is ``inconspicuous'', and this corresponds to the notion of equilibrium as defined in \cite{B,BP}. The reader is referred once more to the papers \cite{B,BP,Wu,Danilova} for more financial as well as mathematical details. Our probabilistic construction of the dynamic bridge $X$ leads to an interesting generalization of such a model, where the signal modelling insider's dynamic information is not necessarily Gaussian. Even in this more general framework, we are able to give an explicit solution for the equilibrium total demand and optimal insider's strategy in our main result of the financial part of the paper, Theorem \ref{t31}. Interestingly, the existence of a solution for insider's maximization problem imposes a very precise structure on the form of the signal volatility $a(t,Z_t)$ (see Section \ref{application}), resulting in the insider's signal being a function of a Gaussian process. This seems to indicate that a non-trivial generalization beyond a Gaussian setting is impossible. Nonetheless, we would like to stress the fact that while the financial application forces the signal to be `almost-Gaussian' as explained in  Remark \ref{rem:AG}, the construction we perform in Section \ref{bridgesec} is much more general, since it includes signals which are not necessarily Gaussian.

The paper is structured as follows: Section \ref{bridgesec} contains the motivation and formal definition of such a bridge.  A new proof of Gaussian bridge construction is given in Section \ref{gaussian}, while in the Section \ref{general} the general construction is proved. At the end of this section, we also give an application of our result to build an Ornstein-Uhlenbeck bridge. Finally, in Section \ref{application} we shortly introduce the financial model and we apply the main result contained  in Section \ref{bridgesec} to find the equilibrium total demand and optimal insider's strategy.

\section{Formulation of the problem and some auxiliary results}\label{bridgesec}
Let $(\Omega , \cG , (\cG_t) , \bbQ)$ be a filtered probability space  satisfying the usual conditions. Note that we do not require $\cG_0$ to be trivial.  Assume that on this probability space there exist two independent standard Brownian  motions, $B$ and $\beta$, and a random variable $Z_0$ in $\cG_0$, which implies that $Z_0$ is independent from $B$ and $\beta$.

Let $\mathcal{Z}=(\Omega , \cG , (\cG_t), (Z_t) ,(P^z)_{z\in \bbR})$ be a diffusion process with values in
$\bbR$. Here we are using the formulation of a Markov process as
given in, e.g., Blumenthal and Getoor \cite{BG} or Sharpe \cite{Sh}. Time varies in the finite interval
$[0,1]$. We will use the notation $\mathbb R_+$ for $[0,\infty)$ and $\cF^Y_t$ for $\sigma(Y_s; s \leq t)$ for any (possibly, vector-valued) stochastic process $Y$.

We further assume that $Z$ is the unique strong solution on $(\Omega , \cG , (\cG_t) , \bbQ)$ of
\begin{equation}\label{SDEsignal} dZ_t = \sigma(t) a(V(t),Z_t) d\beta_t, \quad t\in (0,1],\end{equation}
with $Z_0\in \cG_0$ being a random variable with distribution, $\mu$, and where $\sigma:[0,1]\mapsto \bbR_+$ is a deterministic function, $V(t):=c+\int_0^t\sigma^2(s)\, ds$ for some constant $c>0$, and $a:[0,1]\times \bbR\mapsto \bbR$ is regular enough for ensuring the existence of a unique strong solution. Moreover, $\sigma, V$ and $a$ are required to satisfy further regularity conditions precise statements of which are given in Assumptions \ref{AssZ},  \ref{fs} and \ref{bt} below.

The rest of this section will be devoted to
the construction of a process $X$ and a probability measure $\mu$ on $\bbR$ satisfying the following three conditions:
\begin{itemize}
\item[\bf{C1}] For every $T<1$, $X$ is the unique strong solution of the SDE
\[
X_t=\int_0^t a(s,X_s) dB_s + \int_0^t \alpha(s,X_s, Z_s)ds, \qquad \mbox{for } t \in (0,T]
\]
for some Borel measurable real valued function $\alpha$. Moreover,  $(X,Z)$ is a Markov process. More precisely,  $(\Omega , \cG ,
(\cG_t) , (X_t,Z_t), (P^{x,z})_{(x,z)\in \bbR^2})$ is a
Markov process with values in $\bbR^2$ endowed with its Borel
$\sigma$-algebra, with an initial distribution given by $\delta_{0} \otimes
\mu$ where $\delta_0$ is the Dirac measure at $0$.
 \item[\bf{C2}] $\lim_{t \uparrow 1} X_t$ exists $P^{0,z}$-a.s. and $X_1:=\lim_{t \uparrow 1} X_t=Z_1$, $P^{0,z}$-a.s..
\item[\bf{C3}] $(X_t)_{t \in [0,1]}$ is a local martingale in its own filtration.
\end{itemize}
\begin{remark} In view of Theorem 8.1 in \cite{ls} the condition C3 implies that $X$, if exists, will be a diffusion process with diffusion coefficient $a$ and no drift in its own filtration. As C2 is also in place this diffusion process will be conditioned to hit $Z_1$ at time 1. If we have allowed $X$ to be adapted to the filtration generated by $Z_1$ and $B$ such a process can be obtained using the available theory of `static' Markov bridges (see, e.g., \cite{fpy} and Proposition 37 in  \cite{fabrice}) with a drift $\alpha(t,X_t,Z_1)$ since $Z_1$ and $B$ are independent. However, the condition C1 stipulates that $X$ should be adapted to the filtration generated by the independent processes $Z$ and $B$. This forces us to develop a theory of `dynamic' Markov bridges as we will describe in the subsequent sections.
\end{remark}
\begin{remark}
The main difficulty with the construction of the process $X$ is that all the conditions C1-C3 have to be met simultaneously. To illustrate this point consider the simple case $a=\sigma=1$.

If one allows the drift, $\alpha$,  in condition C1 to depend on $Z_1$, then
\be \label{eq:bb}
X_t= B_t +  \int_0^t \frac{Z_1-X_s}{1-s}\,ds
\ee
has a unique strong solution over $[0,1)$ and its solution can be continuously extended to the full interval $[0,1]$ since conditioned on $Z_1=z$ this is the SDE for a Brownian bridge from $0$ to $z$ over the interval $[0,1]$. This is a Markovian bridge conditioned to hit $Z_1$ at $t=1$ and, moreover, it is a martingale in its own filtration (see expression (10) and the discussion after it in \cite{fwy}).

If we replace $Z_1$ with $Z_t$ in the above formulation we obtain the SDE
\[
X_t= B_t +  \int_0^t \frac{Z_s-X_s}{1-s}\,ds
\]
which has a unique strong solution which satisfies C1 and C2 (see Lemma 2.1 in \cite{fwy}). However, $X$ does not satisfy C3 (see Lemma 2.2 in \cite{fwy}).

One is also tempted to think that a projection of the solution of (\ref{eq:bb}) onto the filtration generated by $X$ and $Z$ could give us the construction that we seek.  Note that the solution of (\ref{eq:bb}) is adapted to the filtration  $(\cF^{B,Z}_t)$ enlarged with $Z_1$. In this enlarged filtration $Z$ has the decomposition
\[
Z_t= \bar{\beta}_t +  \int_0^t \frac{Z_1-Z_s}{1-s}\,ds
\]
where $\bar{\beta}$ is a standard Brownian motion adapted to this filtration (see Theorem 3 in Chap. VI of \cite{Protter}) and independent of $B$. Comparison of the SDEs for $X$ and $Z$ reveals an inherent symmetry of these two processes. Thus, the semimartingale decomposition of these two processes with respect to $(\cF^{X,Z}_t)$ should have a symmetric structure, in particular if one is a martingale with respect to $(\cF^{X,Z}_t)$ so is the other. However, this is inconsistent with the structural assumptions we have on $X$ and $Z$ which are manifested in C1 and (\ref{SDEsignal}).

These examples, in particular the last one, demonstrate that the solution to our problem cannot be obtained via a combination of available enlargement of filtration and nonlinear filtering techniques.
\end{remark}
We would like to stress here that the functions $\sigma$ and $V$ that appeared in the dynamics of $Z$ play a crucial role in the existence of the solution of the problem above. Indeed, suppose that $\sigma \equiv 1$ and $V(t)=t$ for each $t \geq 0$, and  $a(t,z)$ is regular enough to ensure the existence of a square integrable non-constant solution to (\ref{SDEsignal}). Suppose that there exists a solution to the  problem defined by conditions C1-C3.  Consider the probability measure $\bbP$ defined on $(\Om, \cF^X_1 \vee \cF^Z_1)$ by
\[
\bbP(E)=\int_{\bbR} P^{0,z}(E)\, \mu(dz), \qquad \forall E \in \cF^X_1 \vee \cF^Z_1.
\]
Note that $X_1=Z_1, P^{0,z}$-a.s. for every $z \in \bbR$ implies that $X_1=Z_1, \bbP$-a.s., therefore, for any
bounded measurable function $f$,
\be \label{eq:svfc}
\bbE[f(Z_1)]=\bbE[f(X_1)]
\ee
where $\bbE$ is the expectation operator with respect to $\bbP$. One one hand,
\[
\bbE[f(Z_1)]= \bbE[\bbE[f(Z_1)|Z_0]]=\int_{\bbR} E^{P^z}[f(Z_1)]\mu(dz) =\int_{\bbR}\int_{\bbR}f(y) p(z,y) \,dy \, \mu(dz)\]
where $p(z,y)\,dy=P^z(Z_1 \in dy)$.  On the other hand, conditions C1 and C3 imply that
\[
X_t=\int_0^t a(s,X_s)d B^X_s
\]
for some Brownian motion $B^X$ adapted to $\cF^X$. Comparing this to (\ref{SDEsignal}) we see that the law of $X_t$ is that of $Z_t$ conditioned on $Z_0=0$ for any $t \in [0,1]$. Therefore,
\be \label{xp2}
\bbE[f(X_1)]=\int_{\bbR}f(y) p(0,y)\, dy.
\ee
Let $f(y)= e^{i r y}$. Then, in view of (\ref{eq:svfc}) we have
\bean
\int_{\bbR} e^{i r y} p(0,y)\, dy&=& \int_{\bbR}\int_{\bbR}e^{i r y} p(z,y) \,dy \, \mu(dz) \\
&=&\int_{\bbR} e^{i r y}\left(\int_{\bbR} p(z,y)\, \mu(dz)\right)\, dy.
\eean
Note that the interchange of integrals is justified since $|e^{i r y}p(z,y)|< p(z,y)$ and \\ $\int_{\bbR}\int_{\bbR} p(z,y) \,dy \, \mu(dz)=1$. This implies that the characteristic functions of the measures  $p(0,y)\, dy$ and $ \left(\int_{\bbR} p(z,y)\, \mu(dz)\right) \, dy$ are the same. We can assume, without substantial loss of generality, that $p$ is continuous in both parameters\footnote{This can be achieved by standard regularity assumptions on $a$ which will ensure that $p$ is a continuous solution of a Kolmogorov equation (see Theorem 3.2.1 in \cite{SV})}.  Therefore, we can invert the Fourier transform to  identify $\mu$ as the Dirac measure at $0$ and conclude that $X$ and $Z$ have the same law.

However, under the assumption $\mu=\delta_0$, Remark 5.2 (i) in \cite{fwy} shows that such a construction is not possible. Indeed, Remark 5.2 (i) in \cite{fwy} contains the following statement: Given filtration $\mathcal F_t$  and a square integrable $\mathcal F_t$-adapted processes $X$ and $M$ with the same second moments such that i) $X$ is a local martingale in its own filtration, ii) $M$ is an  $\mathcal F_t$-martingale, and iii) $M_1 =X_1$, then  $M_t =X_t$ for all $t\in [0,1]$. Applying this result to our setting we get $X_t = Z_t$ for all $t\in [0,1]$ and, thus,

\[ \int_0 ^t a(s,Z_s) d\beta_s =Z_t=X_t= \int_0 ^t a(s,Z_s) dB_s + \int_0 ^t \alpha(s,Z_s,Z_s) ds.\]
This implies that $\int_0 ^t \alpha(s,Z_s,Z_s) ds$ is a continuous martingale with finite variation, therefore it is identically $0$. Thus,
\[
\int_0 ^t a(s,Z_s) d\beta_s =Z_t=X_t= \int_0 ^t a(s,X_s) dB_s.
\]
Since $B$ and $\beta$ are independent, this yields that $[Z,X]\equiv 0$. However, as $Z=X$, we have $[Z,Z]\equiv 0$, which implies $X=Z\equiv 0$, which is a contradiction.\\

This example highlights that the relationship between $V(t)$ and $t$ is very important for the existence of a solution to the problem we aim to solve. The following assumption formalizes this relationship along with imposing some regularity conditions. In particular, Assumption \ref{AssZ}.1 rules out the above pathology, Assumption \ref{AssZ}.2 controls the speed of convergence of $V(t)-t$ to $0$ as $t \rar 1$ (for an earlier use of this assumption see \cite{Danilova}), and Assumptions \ref{AssZ}.3 to \ref{AssZ}.5 ensure sufficient regularity for the problem.
\begin{assumption}\label{AssZ}Fix a real number $c \in (0,1]$.
 $\sigma : [0,1] \mapsto \bbR_+$ and $a : [0,1]\times \bbR \mapsto \bbR_+$ are two measurable functions such that:\begin{enumerate}
\item $V(t) := c+ \int_0 ^t \sigma^2 (u)du > t$ for every $t\in [0,1)$, and $V(1)=1$;
\item \label{tomodify} $\lim_{t\uparrow 1}\lambda^2(t)\Lambda(t)\log(\Lambda(t))=0$,  where $\lambda (t)=\exp \left\{ -\int_{0}^{t}\frac{1}{V(s)-s}ds\right\} $ and $\Lambda (t)=\int_{0}^{t}\frac{1+\sigma^{2}(s)}{\lambda^{2}(s)}ds$;
\item $\sigma^2 (t)$ is bounded on $[0,1]$;
\item $a(t,z)$ is uniformly bounded away from zero, i.e. there exists a constant $\epsilon >0$ such that $a(t,z) \geq \epsilon$ for all $t \in [0,1]$ and $z\in \bbR$;
\item $a(\cdot, \cdot) \in C^{1,2}$ and has enough regularity in order for (\ref{SDEsignal}) has a unique strong solution\footnote{A sufficient condition ensuring strong solution is given in Assumption \ref{fs}.}.
\end{enumerate}
\end{assumption}

\begin{remark}\label{rem:lambda_to_0}
Notice that Assumption \ref{AssZ}.2 is, in fact, an assumption on $\Lambda(t)$, since we always have that $\lim_{t\uparrow 1}\lambda(t)=0$. Indeed, since $V(t)$ is increasing and $V(1)=1$, we have that $V(t)\leq1$ for $t\in [0,1]$ and therefore
\be
\lambda(t)\leq 1-t
\ee
which leads to conclusion that $\lim_{t\uparrow 1}\lambda(t)=0$. For another use of this assumption and further discussion see \cite{Danilova}.
\end{remark}
Although Assumption \ref{AssZ}.\ref{tomodify} seems to be involved, it is satisfied in many cases. The following remark states a sufficient condition for this assumption to be satisfied.
\begin{remark} Note that when $\lim_{t \uparrow 1} \Lambda(t)< \infty$ the condition is automatically satisfied due to the preceding remark. Next, suppose that $\lim_{t \uparrow 1} \Lambda(t)=\infty$, $\sigma$ is continuous in a vicinity of $1$ and $\sigma(1)\neq 1$. Then, an application of de L'H\^{o}pital rule yields
\[
0\leq \lim_{t \uparrow 1} \frac{\Lambda(t)\log(\Lambda(t))}{\lambda^{-2}(t)}=\frac{1+\sigma^2(1)}{2}\lim_{t \uparrow 1} \frac{\log(\Lambda(t))}{(V(t)-t)^{-1}}.
\]
Then note that since $\lim_{t \uparrow 1}\lambda^2(t)\Lambda(t)=0$, $\Lambda(t)\leq \lambda^{-2}(t)$ for $t$ close to $1$. Thus,
\[
0\leq \lim_{t \uparrow 1} \frac{\Lambda(t)\log(\Lambda(t))}{\lambda^{-2}_t}\leq\frac{1+\sigma^2(1)}{2}\lim_{t \uparrow 1} \frac{\log(\lambda^{-2}(t))}{(V(t)-t)^{-1}}=(1+\sigma^2(1))\lim_{t \uparrow 1} \frac{\int_0^t\frac{1}{V(s)-s}ds}{(V(t)-t)^{-1}}=0,
\]
after another application of de L'H\^{o}pital rule since $\sigma^2(1)\neq 1$. This in particular shows that Assumption \ref{AssZ}.\ref{tomodify} is satisfied when $\sigma$ is a constant.
\end{remark}
Before we present our main result we shall collect some preliminary results on the transition density of the diffusion
\[
d\xi_t=a(t,\xi_t)d\beta_t.
\]
We are in particular interested in the existence and smoothness of this transition density. The natural way to obtain these results is to use the link between the transition density and the fundamental solution of
\be \label{pde0} w_u (u,z)=\half \left(a^2 (u,z) w(u,z)\right)_{zz}, \ee
established in Corollary 3.2.2 in \cite{SV}. However, as we do not assume $a$ to be bounded, this theorem is not applicable. On the other hand, since $a$ is bounded away from $0$, the following function \be
\label{ef} A(t,x):=\int_0^x \frac{1}{a(t,y)}dy \ee
 is well defined and the transformation defined by
$\zeta_t:=A(t, \xi_t)$ will yield, via  It\^{o}'s formula,
\be \label{zeta} d\zeta_t= d\beta_t + b(t,\zeta_t)dt, \ee where \be
\label{bee} b(t,x):=A_t(t,A^{-1}(t,x))-\half a_z(t,A^{-1}(t,x)), \ee
and $A^{-1}$, the inverse of $A$, is taken with respect to the space variable.  This transformation along with the next assumption is going to provide a uniformly elliptic operator which via Theorem 10 in Chap. I of \cite{friedman} and Theorem 3.2.1 of \cite{SV} would imply the existence and smoothness of transition density of $\xi$.
\begin{assumption} \label{fs}
$b$ and $b_x$ are uniformly bounded on $[0,1]\times \bbR$ and $b_x$
is Lipschitz continuous uniformly in $t$.
\end{assumption}

Due to this assumption Corollary 3.2.2 of \cite{SV} implies that the transition density of $\zeta$ is the fundamental solution of
\be \label{eq:pde1}
 w_u(u,z)=\half w_{zz}(u,z)-(b(u,z) w(u,z))_z.
 \ee
For the reader's convenience we recall the definition (p.~3 of \cite{friedman}) of fundamental solution, $\Gamma(t,x;u,z)$, of
(\ref{eq:pde1}) as the function satisfying
\begin{enumerate}
\item For fixed $(t, x)$, $\Gamma(t,x;u,z)$ satisfies (\ref{eq:pde1}) for all $u>t$;
\item For every continuous and bounded  $f:\bbR \mapsto \bbR$
\be \label{eq:fsb}
\lim_{u \downarrow t} \int_{\bbR} \Gamma(t,x;u,z)f(x)dx =f(z).
\ee
\end{enumerate}
The following proposition provides the existence of the transition density of $\xi$ and formalizes the smoothness requirement on the transition density together with some properties that we will use later.
\begin{proposition} \label{Gexists}Under Assumptions \ref{AssZ} and \ref{fs} there
exists a fundamental solution, $\Gamma \in C^{1,2,1,2}$, to (\ref{eq:pde1}) which also solves the adjoint equation
\be \label{eq:adj1}
 v_t(t,x) + b(t,x) v_x(t,x) + \half
v_{xx}(t,x)=0.
\ee Moreover,  the function $G(t,x;u,z)$ defined by
\be \label{GviaGamma} G(t,x;u,z):=\Gamma(t, A(t,x);u , A(u,z))\frac{1}{a(u,z)},
\ee
satisfies (\ref{pde0}) for fixed $(t,x)$ and it is the transition density of $\xi$, i.e.
\[
G(t,x;u,z) dz= P(\xi_u \in dz|\xi_t=x) \qquad \mbox{for } u \geq t.
\]
Furthermore, $G_x(t,x;u,z)$ exists and satisfies
\be \label{nodrift}
\int_{\bbR}G_x(t,x;u,z)dz=0=\int_{\bbR}\Gamma_x(t,x;u,z)dz.
\ee
\end{proposition}
 \begin{proof} Since $b$ and $b_x$ are bounded and H\"older continuous under the assumptions of the proposition, it follows from Theorem 10 in Chap. I of \cite{friedman} that the fundamental solution, $\Gamma(t,x;u,z)$, to  (\ref{eq:pde1}) exists  and is also the fundamental solution of (\ref{eq:adj1}) by Theorem 15 in Chap. I of \cite{friedman}. In particular,  $\Gamma \in C^{1,2,1,2}$.
 Moreover, Assumption \ref{fs} also implies, due to Corollary 3.2.2 in \cite{SV},  that $\Gamma$ is the transition density
 of $\zeta$.

Define $G(t,x;u,z)$ by (\ref{GviaGamma}) and observe that $G(t,x;u,z)$ for fixed $(t,x)$ solves (\ref{pde0}). Since by definition $\zeta_t=A(t,\xi_t)$ and $A$ is strictly increasing
\bean
G(t,x;u,z)\, dz& =&\Gamma(t,A(t,x);u,A(u,z))\frac{1}{a(u,z)}\, dz \\
&=&\Gamma(t,A(t,x);u,A(u,z))\, dA(u,z)\\
&=& P(\zeta_u \in dA(u,z)|\zeta_t=A(t,x))\\
&=&P(A(u, \xi_u) \in dA(u,z)|A(t,\xi_t)=A(t,x)) \\
&=&P(\xi_u \in dz|\xi_t=x),
\eean
which establishes that $G$ is the transition density of $\xi$.

Moreover, equations (6.12) and (6.13) following Theorem 11 in Chapter I of \cite{friedman}
give the following estimates:
 \bea \label{est0}
\Gamma(t,x;u,z) &\leq& C
\frac{1}{\sqrt{u-t}}\exp\left(-c\frac{(x-z)^2}{2
 (u-t)}\right), \mbox{ and} \\
 \left|\Gamma_x(t,x;u,z)\right| &\leq& C \frac{1}{u-t}\exp\left(-c\frac{(x-z)^2}{2
 (u-t)}\right), \label{est1}
 \eea
 for some positive $C$, depending on $c$, and any $c <1$.  Note that $\int_{\bbR}G(t,x;u,z)dz=1=\int_{\bbR}\Gamma(t,x;u,z)dz$  since both $G$ and $\Gamma$ are transition densities. Thus, (\ref{nodrift}) will hold if one can  interchange the derivative and the integral. This is justified since  due to (\ref{est1}) $|\Gamma_x| \leq K \exp\left(-c_1 z^2 \right)$ when $x$ is restricted to a bounded interval and where the constants $K$ and $c_1$ do not depend on $x$ and might depend on $u$ and $t$. The result then follows from  an application of the Dominated Convergence Theorem.
\end{proof}\\

In order to motivate our main result let's first consider the special case of $\sigma\equiv 0$ so that $Z_1=Z_0$. In terms of the insider trading models that we have in mind this corresponds to the case when the insider has the complete information at time-$0$ regarding the time-$1$ value of the traded asset as in \cite{B}.  If $\mu(dz)=G(0,0; 1, z)dz$, then there exists a unique strong solution to
 \[
 dX_t=a(t,X_t)dB_t + a^2(t,X_t) \frac{G_x(t, X_t, 1, Z_0)}{G(t, X_t; 1,
 Z_0)}dt,
 \]
 with the initial condition $X_0=0$, which satisfies all the properties stated in {\bf C1-C3} (see Proposition 37 in  \cite{fabrice} for a proof of this and other related results. Observe that this result does not require uniform ellipticity of $a$, thus the extension to time inhomogeneous case is immediate).

 The specific form of the drift term in the SDE above thus gives us a hint to formulate the solution of the original problem stated at the beginning of this section. Before we state our main result we introduce one last assumption.

\begin{assumption} \label{bt} $b(t,x)$ is absolutely continuous with respect to $t$ for each $x$, i.e. there exists a  measurable function $b_t:[0,1] \times  \bbR \mapsto \bbR$ such that
\[
b(t,x)=b(0,x)+ \int_0^t b_t(s,x)\, ds,
\]
for each $x \in \bbR$. Moreover, $b_t$ is uniformly bounded.
\end{assumption}
 \begin{theorem} \label{mainresult} Suppose $\mu(dz)=G(0,0; c, z)dz$ where $c \in (0,1)$ is
 the real number fixed in Assumption \ref{AssZ} and $G$ is given by (\ref{GviaGamma}). Let for $t < 1$
 \be \label{SDEx}
 dX_t= a(t,X_t)dB_t + a^2(t,X_t) \frac{\rho_x (t, X_t,
 Z_t)}{\rho(t,X_t,Z_t)}dt,
\ee where
\be \label{d:rho}
\rho(t,x,z):=G(t,x;V(t),z).
\ee
 Under Assumption \ref{AssZ}, \ref{fs} and \ref{bt}, on every interval $[0,T]$ with $T <1$, there exists a
unique strong solution to the above SDE with the initial condition $X_0=0$.
Moreover, the conditions {\bf C1-C3} are satisfied.
\end{theorem}
We will give a proof of this result in the special case $a \equiv 1$ in Section \ref{gaussian} and a proof of the general result will be given  in Section \ref{general}. However, we shall now state and prove two lemmata to show how the choice of the drift term in (\ref{SDEx}) would imply condition {\bf C3}, i.e. $X$ as defined in (\ref{SDEx}) is a local martingale in its own filtration. Before we can formulate them,  we need to introduce the following notation.

Let $\cF:=\sigma(X_t,Z_t; t < 1)$ and define the probability measure $\bbP$ on $(\Om, \cF)$  by
\be \label{d:bbP}
\bbP(E)=\int_{\bbR} P^{0,z}(E) \mu(dz),
\ee
for any $E \in \cF$.  Of course, in order for this construction to make sense we need the existence of a solution to (\ref{SDEx}). This will be proved in Section \ref{general} in Corollary \ref{xexists}, thus, the above probability space exists and is well defined.  Under $\bbP$, $(X_t,Z_t)_{t < 1}$ would still be a strong Markov process (see Corollary \ref{xexists}). Let $\mathcal{N}$ be the null sets of $\bbP$. Proposition 2.7.7 in \cite{ks}  shows that  the filtration $(\mathcal{N} \vee \cF^{X,Z}_t)_{t <1}$ is right-continuous.  With an abuse of notation we shall still denote  the $\sigma$-algebra generated by $\cF$ and $\mathcal{N}$ with  $\cF$, and  denote $\mathcal{N} \vee \cF^Y_t$ with $\cF^Y_t$ for any $(\cF^{X,Z}_t)_{t \leq 1}$-adapted process $Y$. Next, let $\widetilde{\cF}^X_t:=\cap_{1>u>t}{\cF}^X_u$. We shall see in Remark \ref{X-sm} later that the $\cF^X$ is right-continuous, i.e. $\cF^X_t=\widetilde{\cF}^X_t$. We say that $g_t: \Om \times \bbR \mapsto \bbR$ is the conditional density of $Z_t$ given $\cF^X_t$, if $g_t$ is measurable with respect to the product $\sigma$-algebra, $\cF_t^X \times \mathcal{B}$ where $\mathcal{B}$ is the Borel $\sigma$-algebra of $\bbR$, and for any bounded measurable function $f$
\[
\bbE[f(Z_t)|\cF^X_t]=\int_{\bbR}f(z)g_t(\om, z)\,dz,
\]
where $\bbE$ is the expectation operator under $\bbP$. Note that due to Markov property of $(X,Z)$, $\bbE^{\bbQ}[f(Z_1)|\cF^X_t]=\bbE[f(Z_1)|\cF^X_t]$ but we will keep the above notation for the clarity of the exposition. We will often write $\bbP[Z_t \in dz| \cF^X_t]=g_t(\om,z)\,dz$ in order to refer to the conditional density property described above.  Now, we are ready to state and prove the following lemma.
\begin{lemma} \label{lmart} Suppose there exists a  unique strong solution of (\ref{SDEx}). If $\rho(t,X_t, \cdot)$ given by (\ref{d:rho}) is the conditional density of $Z_t$ given $\cF^X_t$ for every $t \in [0,1)$, then $(X_t)_{t \in [0,1)}$ is a local martingale in its own filtration.
\end{lemma}
\begin{proof}
It follows from standard filtering theory (e.g. Theorem 8.1 in \cite{ls}) that
\[
dX_t= a(t,X_t)dB^X_t+ a^2(t,X_t)\bbE\left[\frac{\rho_x (t, X_t,
 Z_t)}{\rho(t,X_t,Z_t)}\bigg | \cF^X_t\right] dt,
\]
where $B^X$ is an $\cF^X-$Brownian motion. However, if $\rho(t, X_t, \cdot)$ is the conditional density of $Z_t$,
\[
 \bbE\left[\frac{\rho_x (t, X_t,
 Z_t)}{\rho(t,X_t,Z_t)}\bigg | \cF^X_t\right]=\int_{\bbR} G_x(t, X_t; V(t), z) dz=0
\]
due to Proposition \ref{Gexists}, so that
\[
dX_t= a(t,X_t)dB^X_t\]
and, thus, $X$ is a local martingale since $a$ is continuous.
\end{proof}\\

In view of this lemma we show in Section \ref{general} that $\rho(t,X_t, \cdot)$  is indeed the conditional density of $Z_t$ given $\cF^X_t$ for every $t \in [0,1)$. The following lemma will be key in proving this result.
\begin{lemma} \label{pequivrho} Suppose there exists a unique strong solution of (\ref{SDEx}). Let $U_t:=A(V(t),Z_t)$ and $R_t:=A(t,X_t)$, where $A$ is defined by (\ref{ef}). Define
\be \label{rhoviap}
p(t,x,z):=\rho(t, A^{-1}(t,x), A^{-1}(V(t),z)) a(V(t), A^{-1}(V(t),z)),
\ee
where $\rho$ is given by (\ref{d:rho}). Then,
\begin{enumerate}
\item $p(t, R_t, \cdot)$ is the conditional density of $U_t$ given $\cF^R_t$ iff $\rho(t,X_t, \cdot)$  is  the conditional density of $Z_t$ given $\cF^X_t$.
\item $(p(t,R_t, \cdot))_{t \in [0,1)}$ is a weak solution to the following stochastic PDE:
\bea
\label{spde0}
&& g_t(z)= \Gamma(0,0;c,z)+ \int_0^t \sigma^2(s) \left\{-(b(V(s),z) g_s(z))_z + \half (g_s(z))_{zz}\right\} ds \\
&& +\int_0^t   g_s(z) \left(\frac{p_x (s, R_s,
 z)}{p(s,R_s,z)}- \int_{\bbR} g_s(z) \frac{p_x (s, R_s,
 z)}{p(s,R_s,z)}dz\right)dI^g_s,  \nn
 \eea
 where
 \[
dI^g_s= dR_s - \left(\int_{\bbR}\left[\frac{p_x}{p}(s, R_s,z)+b(s,R_s)\right] g_s(z)dz\right)ds.
\]
\end{enumerate}
\end{lemma}
\begin{proof}
Notice that since $A(t, \cdot)$ is strictly increasing, $\cF^R_t=\cF^X_t$ for every $t \in [0,1)$ and there is a one-to-one correspondence between the conditional density of $Z$  and that of $U$. More precisely,
\[
\bbP[Z_t \in dz |\cF^X_t]=\bbP[U_t \in dA(V(t),z)|\cF^R_t].
\]
Thus, if $\bbP[Z_t \in dz |\cF^X_t]=\rho(t, X_t,z)\,dz$, then
\bean
\bbP[U_t \in dz |\cF^R_t]&=&\bbP[Z_t \in dA^{-1}(V(t),z)|\cF^X_t] \\
&=&\rho(t, X_t, A^{-1}(V(t),z)) dA^{-1}(V(t),z) \\
&=&\rho(t, X_t, A^{-1}(V(t),z))a(V(t), A^{-1}(V(t),z))\, dz\\
&=&\rho(t, A^{-1}(t,R_t), A^{-1}(V(t),z))a(V(t), A^{-1}(V(t),z))\,dz\\
&=&p(t,R_t,z)dz
\eean by (\ref{rhoviap}). The reverse implication can be proved similarly.

In order to prove the second assertion observe that due to (\ref{GviaGamma}) and (\ref{rhoviap}) we have
\be \label{pviaGamma}
p(t,x,z)=\Gamma(t,x;V(t),z).
\ee
We have seen in Proposition \ref{Gexists} that $\Gamma(t,x;u,z)$ solves (\ref{eq:pde1}) for fixed $(t,x)$ and it also solves (\ref{eq:adj1}) for fixed $(u,z)$. Combining these two facts yields that $p$ satisfies
 \be \label{pderho}
 p_t(t,x,z) + b(t,x) p_x(t,x,z) +\half p_{xx}(t,x,z)= -\sigma^2(t) (b(V(t),z) p(t,x,z))_z + \half \sigma^2(t) p_{zz}(t,x,z).
\ee
Using It\^{o}'s formula and (\ref{pderho}), we get
\bean
p(t,R_t, z)&=&\Gamma(0,0;c,z)+ \int_0^t \sigma^2(s) \left\{-(b(V(s),z) p(s,R_s,z))_z + \half (p(s,R_s,z))_{zz}\right\} ds \\
&&+\int_0^t p_x(s,R_s, z) [dR_s - b(s,R_s) ds]
\eean
Due to (\ref{nodrift}), $dI^g_s= dR_s - b(s,R_s)ds$ when $g_t(z) =p(t,R_t, z)$. By repeating this argument we arrive at the desired conclusion.
\end{proof}\\

Before we give a proof of Theorem \ref{mainresult}, we will first investigate the Gaussian case, i.e. $a \equiv 1$.

\section{Gaussian case} \label{gaussian}

Under the assumption $a \equiv 1$, $Z$ becomes a Gaussian martingale and its transition density is given by $G(t,x;u,z)=\frac{1}{\sqrt{2 \pi (u-t)}}\exp(-\frac{(x-z)^2}{2 (u-t)})$ since it is a time-changed Brownian motion where the time-change is deterministic.  In
this case, the equation (\ref{SDEx}) reduces to
\be \label{sde:g}
dX_t=dB_t+ \frac{Z_t-X_t}{V(t)-t}dt.
\ee
This equation along with various properties of its
solution is discussed in Danilova \cite{Danilova}, F\"ollmer, et al.
\cite{fwy} and Wu \cite{Wu}.
\begin{theorem} \label{t:gaussian} Suppose $a\equiv 1$ and $\rho$ is given by (\ref{d:rho}). Then, Theorem \ref{mainresult} holds.
\end{theorem}

The proof of the above theorem will be done in several steps, first of which being the following proposition.
\begin{proposition} \label{c1:g} There exists a unique strong solution to (\ref{sde:g}) over [0,1). Moreover, $((X_t,Z_t))_{t \in [0,1)}$ is strong Markov.
\end{proposition}
\begin{proof}
Since $\frac{z-x}{V(t)-t}$ is Lipschitz over any $[0,T]$ for $T <1$,
there exists a unique strong solution to the above equation with
$X_0=0$ by Theorem 38 of Chap. V in \cite{Protter}.  Moreover, Theorem 5.4.20 in \cite{ks} yields $(X,Z)$ has strong Markov property.
\end{proof}\\

The above proposition shows that condition {\bf C1} of the bridge construction is satisfied. We next show that the solution to (\ref{sde:g}) satisfies condition {\bf C2} and then conclude this section with a proof of Theorem \ref{t:gaussian}.
\begin{lemma}
Let Assumption \ref{AssZ} hold. Let $\lambda (t)=\exp \left\{
-\int_{0}^{t}\frac{1}{V(s)-s}ds\right\} $ and $\Lambda
(t)=\int_{0}^{t}\frac{1+\sigma^2(s)}{\lambda^{2}(s)}ds$ be as in Assumption \ref{AssZ}, and $\ell >0$ be the associated constant
in Assumption \ref{AssZ}. Define
\be \varphi(t,x,z)=\frac{1}{\sqrt{2(\Lambda
(t)+\ell)}}e^{\frac{(x-z)^2}{2 \lambda^2(t)(\Lambda(t)+\ell)}}
\ee
Then, $(\varphi(t,X_t,Z_t))_{t\in [0,1)}$ is a positive supermartingale and  \be
\label{zero_varphi} \lim_{t\uparrow 1} \varphi (t,x,z) =
+\infty,\quad x\neq z \ee \end{lemma}
\begin{proof}  Direct calculations give
\be \label{rem:phi_PDE}
\varphi_t(t,x,z)+\frac{z-x}{V(t)-t}\varphi_x(t,x,z)+\frac{1}{2}\varphi_{xx}(t,x,z)+\frac{\sigma^2(t)}{2}\varphi_{zz}(t,x,z)=0
\ee
Thus, it follows from  It\^o's formula that $\varphi(t,X_t,Z_t)$ is a local martingale. Since it is obviously positive, it is a supermartingale.

In order to prove the convergence in the case of $x\neq z$, consider two cases:
\begin{itemize}
\item Case 1: $\lim_{t\uparrow 1}\Lambda(t)< +\infty$. Then, since due to the Remark \ref{rem:lambda_to_0} we have that $\lim_{t\uparrow 1}\lambda(t)=0$, we obtain that
    \be
    \lim_{t\uparrow 1}\varphi(t,x,z)=\lim_{t\uparrow 1}\frac{1}{\sqrt{2(\Lambda
(t)+\ell)}}e^{\frac{(x-z)^2}{2 \lambda^2(t)(\Lambda(t)+\ell)}}=+\infty
    \ee
    \item Case 2: $\lim_{t\uparrow 1}\Lambda(t)= +\infty$. In this case we will have:
    \[
    \lim_{t\uparrow 1}\log\varphi(t,x,z)=\lim_{t\uparrow 1}\log(2(\Lambda
(t)+\ell))\left[\frac{(x-z)^2}{2 \lambda^2(t)(\Lambda(t)+\ell)\log(2(\Lambda
(t)+\ell))}-\frac{1}{2}\right]=+\infty
    \]
    where the last equality is due to Assumption \ref{AssZ}.\ref{tomodify}. Indeed, the condition yields that
    \[
    \lim_{t\uparrow 1}\lambda^2(t)2(\Lambda(t)+\ell)\log(2(\Lambda(t)+\ell))=0
     \]
     since when $\lim_{t \uparrow 1}\Lambda(t)=\infty$, $\lim_{t \uparrow 1}\frac{\log(2(\Lambda(t)+\ell))}{\log(\Lambda(t))}=1$ and $\lim_{t\uparrow 1}\lambda^2(t)\Lambda(t)=0$. Therefore we have $\lim_{t\uparrow 1}\varphi(t,x,z)=+\infty$
\end{itemize}
The proof is now complete. \end{proof}

\begin{proposition} \label{bpconv}  $P^{0,z}(\lim_{t \uparrow 1}
X_t=Z_1)=1$ where $X$ is the unique strong solution to (\ref{sde:g}).
\end{proposition}
\begin{proof} Let $M_t:= \varphi(t,X_t,Z_t)$. Then, $M=(M_t)_{t \in [0,1)}$ is a positive supermartingale by the previous lemma. Using the supermartingale convergence theorem, there exists an $M_1 \geq 0$
such that $\lim_{t \uparrow 1} M_t=M_1,$ $P^{0,z}$-a.s.. Using Fatou's lemma and the fact that $M$ is a supermartingale, we have
\[
M_0 \geq \liminf_{t \uparrow 1}E^{0,z}[M_t] \geq E^{0,z}[M_1]=E^{0,z}\left[\lim_{t \uparrow 1}\varphi(t,X_t,Z_t)\right],
\]
where $E^{0,z}$ is the expectation operator with respect to $P^{0,z}$.
Since $M_0$ is finite, one has  $\lim_{t\uparrow 1} \varphi(t,X_t,Z_t)$ is finite $P^{0,z}$-a.s.. Therefore,  $P^{0,z}(\lim_{t\uparrow 1}  X_t \neq Z_1)=0$ in view of  (\ref{zero_varphi}).
\end{proof}
\begin{proposition} \label{c3:g} Let $a\equiv 1$, $\mu(dz)=G(0,0; c, z)dz$ where $c \in (0,1)$ is
 the real number fixed in Assumption \ref{AssZ} and $G$ is given by (\ref{GviaGamma}). Then,
 \[
 \bbP[Z_t \in dz|\cF^X_t]=G(t,X_t;V(t),z)\,dz.
 \]
 \end{proposition}
 \begin{proof} As $(X,Z)$ is jointly Gaussian the conditional distribution of $Z_t$ given $\cF^X_t$ is also Gaussian (see Theorem 11.1 in \cite{ls}). Thus, it suffices to find the conditional mean $\widehat{Z}_t$ and the variance $\gamma_t$ in order to characterize the distribution completely. Theorem 10.3 in  \cite{ls} yields
 \be \label{e:cmean}
 d\widehat{Z}_t=\frac{\gamma_t}{V(t)-t}\left\{dX_t- \frac{\widehat{Z}_t -X_t}{V(t)-t}dt\right\},
 \ee
 and
 \be \label{e:cvar}
 \frac{d\gamma_t}{dt}=\sigma^2(t)-\frac{\gamma_t}{V(t)-t},
 \ee
 with the initial conditions that $\widehat{Z_0}=\bbE[Z_0]=0$ and $\gamma_0=c$ due to the choice of $\mu$. In particular, $\gamma_t$ is deterministic. One can verify directly that $\gamma_t=V(t)-t$ satisfies (\ref{e:cvar}) and the initial condition since $V(0)=c$ by Assumption \ref{AssZ}. Thus, (\ref{e:cmean}) becomes
 \[
 d\widehat{Z}_t=dX_t- \frac{\widehat{Z}_t -X_t}{V(t)-t}dt,
 \]
 i.e. $\widehat{Z}$ is a solution to the following SDE:
\be \label{e:cmean2}
 dY_t=dX_t- \frac{Y_t -X_t}{V(t)-t}dt,
 \ee
Clearly, choosing $Y_t=X_t$ will solve this SDE. Moreover, as the function $\frac{y-x}{V(t)-t}$ is Lipschitz on $[0,T]$ for any $T <1$, it follows from Theorem 7 in Chap. V of \cite{Protter} that (\ref{e:cmean2}) has a unique solution. Thus, $\widehat{Z}_t=X_t$, which in turn yields that the conditional distribution of $Z_t$ is Gaussian with mean $X_t$ and variance $V(t)-t$. Note that the density associated to this distribution is given by $G(t, X_t; V(t), \cdot)$ when $a \equiv 1$.
 \end{proof}\\

\noindent {\sc Proof of Theorem \ref{t:gaussian}.}\hspace{3mm} Propositions \ref{c1:g} and \ref{bpconv} establish that conditions {\bf C1} and {\bf C2} are satisfied. Finally, {\bf C3} is satisfied as well due to Proposition \ref{c3:g} in view of Lemma \ref{lmart}.
\qed

\section{The general case}\label{general}
We now go back to proving Theorem \ref{mainresult}. The proof is structured in several steps in the following way. We  first show that there exists a strong solution, which is also Markov, to the system of SDEs given by (\ref{SDEsignal}) and (\ref{SDEx}) on the time interval $[0,1)$. Then we show that $\lim_{t \uparrow 1} X_t$ exists and equals $Z_1$, $P^{0,z}$-a.s.  implying that there is no explosion until time $1$ so that the solution can be continuously extended to the whole interval $[0,1]$ and satisfies the bridge condition. Then we characterize the conditional distribution of $Z_t$ given $\cF^X_t$ and identify it with  $\rho(t,X_t, \cdot)$ which will in turn imply that $X$ is a local martingale in its own filtration via Lemma \ref{lmart}. Finally, we provide an application of our method to the construction of Ornstein-Uhlenbeck bridges.

\subsection{Existence of a strong solution on the time interval $[0,1)$ and the bridge property}
Recall from Lemma \ref{pequivrho} that $U_t=A(V(t),Z_t)$, $R_t=A(t,X_t)$ where $A$ is defined in (\ref{ef}) and $\rho$ is related to $p$ via (\ref{rhoviap}). Since $A$ is strictly increasing, the existence of the strong  solution with Markov property to the system of SDEs given by (\ref{SDEsignal}) and (\ref{SDEx}) and the convergence of $X_t$ to $Z_1$ is equivalent to the existence of a strong solution with a Markov property of the following system, which can be obtained by an application of It\^o's formula, \bea
dU_t&=&\sigma(t)d\beta_t +\sigma^2(t)b(t,U_t)dt  \nn \\
dR_t&=&dB_t + \left\{ \frac{p_x (t, R_t, U_t)}{p(t,R_t,U_t)} +
b(t,R_t)\right\}dt, \label{sdeR} \eea
and convergence of $ R_t$ to $U_1$.

 First observe that due to (\ref{pviaGamma}), we have $\frac{p_x (t, x, z)}{p(t,x,z)} +
b(t,x)=\frac{\Gamma_x(t,x;V(t),z)}{\Gamma(t,x;V(t),z)}+b(t,x)$. Thus, due to Lemma \ref{Gambounds} and Assumption \ref{fs}, $\frac{p_x (t, x, z)}{p(t,x,z)} +
b(t,x)$ is locally  Lipschitz for $t \in [0,T]$ for any $T <1$\footnote{Lemma \ref{Gambounds} gives a lower bound for $p$. This implies that $\frac{p_x}{p}(t,x,z)$ has locally bounded derivatives with respect to $x$ and $z$ since $\Gamma$ has continuous second derivatives. Moreover, $b$ is Lipschitz by assumption, thus the claim holds. Note that we require $T<1$ so that $V(t)-t$ is bounded away from $0$ for $t \in [0,T]$.} and therefore $R$ is the
unique strong solution to the corresponding stochastic differential
equation on $[0,1)$ up to an explosion time (see Theorem 38 of Chap. V in \cite{Protter}).  Moreover, the solution will have strong Markov property for any stopping time strictly less than the explosion time by Theorem 5.4.20 in \cite{ks}.  We shall now see that there won't be any explosion until time $1$ and, indeed, $R_t$ converges to $U_1$.

\begin{proposition} \label{gconv} Suppose that Assumptions \ref{AssZ}, \ref{fs} and \ref{bt} are satisfied.
Then \[P^{0,z}(\lim_{t \uparrow 1} R_t= U_1)=1.\]
 \end{proposition}
 \begin{proof}
As observed before there exists a strong solution to  (\ref{sdeR}) up to an explosion time. Let's denote this explosion time with $\tau$. We will first argue that $P^{0,z}(\tau<1)=0$. Recall that $p(t,x,z)=\Gamma(t,x;V(t),z)$ (see (\ref{pviaGamma})) and define \be \label{defh}
h(t,x;u,z):=\frac{\Gamma(t,x;u,z)}{q(u-t,x,z)}
\ee
where $q$ is the transition density of a standard Brownian motion. This yields that
\[
\frac{p_x(t,x,z)}{p(t,x,z)}=\frac{z-x}{V(t)-t}+\frac{h_x(t,x;V(t),z)}{h(t,x; V(t),z)}.
\]
Proposition \ref{hxlim} in the Appendix states that $\frac{h_x(t,x;V(t),z)}{h(t,x; V(t),z)}$ is uniformly bounded on $[0,1]\times\bbR \times [0,1]\times\bbR$, thus we can define an equivalent probability measure $\overline P^{0,z}$ on $\cG_1$ (recall that all the stochastic processes have been defined on $(\Om, \cG, (\cG_t))$ which is introduced at the beginning of Section \ref{bridgesec}) by the following:
\[
\frac{d\overline P^{0,z}}{dP^{0,z}} := \mathcal E\left(\int_0 ^\cdot \sigma(s)b(s,U_s)d\beta_s\right)_1 \mathcal E\left(\int_0 ^\cdot \left\{b(s,R^{\tau}_s)+\frac{h_x(s,R^{\tau}_s;V(s),Z_s)}{h(s,R^{\tau}_s; V(s),Z_s)}\right\}dB_s\right)_1,
\]
where $R^{\tau}$  is the process $R$ stopped at $\tau$ and $\mathcal E(\cdot)_t$ denotes the Dol\'eans-Dade stochastic exponential taken at time $t$. Recall that $\sigma$ and $b$ are bounded by Assumptions \ref{AssZ} and \ref{fs}.
Under this new measure $\overline P^{0,z}$, (\ref{sdeR}) becomes
\begin{eqnarray}
dU_t &=& \sigma(t)d\overline \beta_t \nn\\
dR^{\tau}_t &=& d\overline B_t+\frac{p_x(t,R^{\tau}_t,U_t)}{p(t,R^{\tau}_t,U_t)}dt=d\overline B_t+\frac{U_t-R^{\tau}_t}{V(t)-t}dt, \label{amc}
\end{eqnarray}
where $\overline \beta$ and $\overline B$ are two $(\cG_t)$-Brownian motions under $\overline P^{0,z}$. First, observe that if $P^{0,z}(\tau<1)>0$, so is $\overline P^{0,z}(\tau<1)$. However, we have shown in Proposition \ref{c1:g} that  (\ref{amc}) has a non-exploding solution over $[0,1)$. This contradiction implies that  $P^{0,z}(\tau<1)=0$ . Moreover, Proposition \ref{bpconv} shows that $\overline P^{0,z}(\lim_{t \uparrow 1} R_t=U_1)=1$. This yields  $P^{0,z}(\lim_{t \uparrow 1} R_t=U_1)=1$ due to the equivalence of the measures.

 \end{proof}\\

 The next proposition and its corollary sum up what we have achieved so far in this section.
\begin{proposition} \label{m-estimate} Under Assumptions \ref{AssZ}, \ref{fs} and \ref{bt}, there exists a unique strong solution to the SDE
\[
dR_t=dB_t + \left\{ \frac{p_x (t, R_t, U_t)}{p(t,R_t,U_t)} +
b(t,R_t)\right\}dt,  \]
with $R_0=0$ such that $(R,U)$ has strong Markov property over the interval $[0,1)$. Moreover, $P^{0,z}(\lim_{t \uparrow 1}R_t=U_1)=1$.
\end{proposition}
\begin{proof} Propositions \ref{gconv} and \ref{hxlim} imply that $P^{0,z}(\lim_{t \uparrow 1}R_t=U_1)=1$, which in turn yields that there is no explosion until time $1$. Thus, Theorem 38 of Chap. V in \cite{Protter} gives that $R$ is indeed the strong solution of the SDE over the time interval $[0,1)$. Moreover,  due to Theorem 5.4.20 in \cite{ks} $(R,U)$ has strong Markov property.
 \end{proof}\\

 The following corollary is immediate due to the one-to-one relationship, via the strictly monotone transformation $A$, between $(R,U)$ and $(X,Z)$.
\begin{corollary} \label{xexists} Under Assumptions \ref{AssZ}, \ref{fs} and \ref{bt}, there exists a unique strong solution to (\ref{SDEx}) on $[0,1]$ such that $(X,Z)$ has a  strong Markov property. Moreover, $\lim_{t \uparrow 1} X_t$ exists $P^{0,z}$-a.s. and $X_1:=\lim_{t \uparrow 1} X_t=Z_1$, $P^{0,z}$-a.s..
\end{corollary}
\subsection{{Conditional distribution of  $Z$}}
We now turn to proving $\rho(t,X_t, \cdot)$ is the conditional density of $Z_t$ given $\cF^X_t$, which will in turn imply that the solution of (\ref{SDEx}) is a local martingale in its own filtration via Lemma \ref{lmart}. In order to find the conditional density of $Z$ we will first  find the conditional density of $U$ given $(\cF^R_t)$ and then use Lemma  \ref{pequivrho}. The reader is asked to review  the notation introduced after the statement of Theorem \ref{mainresult} at this point. Recall that $U_t=A(V(t),Z_t)$ and $R_t=A(t,X_t)$ where $A$ is the function defined by (\ref{ef}).  Under this transformation $U_0$ has a probability density given by $ \bbP(U_0 \in dz) =\Gamma(0, 0; c, z)\,dz$, where the measure $\bbP$ is defined by (\ref{d:bbP}).

Next,  fix a $T <1$ and  let $\bbP_T:=\bbP|_{\cF^{X,Z}_T}$ be the restriction of $\bbP$ to $\cF^{X,Z}_T$. The reason for this restriction is due to the fact that the drift term in (\ref{sdeR}) is not defined at $t=1$ and this will lead to inapplicability of the results of \cite{ko} that we cite later in this subsection.  Note that $\bbP[U_t \in dz |\cF^R_t]= \bbP_T[U_t \in dz |\cF^R_t]$ for $t \in [0,T]$ and, since $T$ is arbitrary, this identity will allow us to obtain all the conditional distributions of $Z_t$ for $t<1$.

The remainder of this subsection is devoted to the proof of that
\[
\bbP_T[U_t\in dz|\cF^R_t]=p(t,R_t,z)\, dz,
\]
where $p$ is defined by (\ref{rhoviap}). In order to achieve this goal we will use the characterization of the conditional distributions obtained by Kurtz and Ocone \cite{ko}. We refer the reader to \cite{ko} for all unexplained details and terminology.
\begin{remark} \label{X-sm} Let $\overline{\bbP}_T$ be the absolutely continuous measure on the same space defined by the Radon-Nikod\'ym derivative \[
\exp\left\{-\half\int_0^T  \left( \frac{p_x (s, R_s, U_s)}{p(s,R_s,U_s)} +
b(s,R_s)\right)^2ds - \int_0^T  \left( \frac{p_x (s, R_s, U_s)}{p(s,R_s,U_s)} +
b(s,R_s)\right)dB_s\right\}.\]
Note that, under $\overline{\bbP}_T$, $R$ is a Brownian motion independent of $U$. Moreover,  as  we have just seen, there is no explosion before time $1$ for the system of SDEs for $(U,R)$.  Thus, it follows from the no-explosion criterion (see Exercise 2.10 in Chap. IX of \cite{RY}) that $\overline{\bbP}_T$ is a probability measure equivalent to $\bbP_T$. As the natural filtration of a Brownian motion is right continuous, this in turn implies that $(\cF^X_t)_{t\in[0,T]}$ is right continuous, too.
\end{remark}
Let $\cP$ be the set of probability measures on the Borel sets of $\bbR$ topologized by weak convergence.  Given $m \in \cP$ and $m-$integrable $f$ we write $m f:=\int_{\bbR} f(z) m(dz)$.  The next result is Lemma 1.1. in \cite{ko}:
\begin{lemma} There is a $\cP$-valued $\cF^X$-optional process $\pi_t(\om, dx)$ such that
\[
\pi_t f= \bbE[f(U_t)|\cF^R_t]
\]
for all bounded measurable $f$. Moreover, $\pi_t$ has a right continuous version.
\end{lemma}
Let's recall  the {\em  innovation process}
\[
I_t= R_t - \int_0^t \pi_s \kappa_s ds
\]
where $\kappa_s(z):= \frac{p_x (s, R_s, z)}{p(s,R_s,z)} +b(s,R_s)$. The next lemma will show that $R$ is an integrable process and, thus, $\pi_s \kappa_s$ exists for all $s <1$ since $Z$ is integrable and $\frac{p_x (s, x, z)}{p(s,x,z)} = \frac{z-x}{V(s)-s} +$  a bounded function, due to Proposition \ref{hxlim}.
\begin{lemma} Let $R$ be the unique strong solution of (\ref{sdeR}) under Assumptions \ref{AssZ}, \ref{fs} and \ref{bt}. Then, for every $T<1$,
\[
\bbE\left[R_t^2\right] \leq  C(1+\nu_T^{-2}) e^{C \nu_T^{-2} t},
\]
for every $t \leq T$ where $C$ is a constant and $\nu_T:=\inf_{t \leq T}(V(t)-t)$.
\end{lemma}
\begin{proof} Note that
 \bean
\bbE\left[R_t^2\right] &\leq& C \left(t + \int_0^t \bbE\left(\frac{p_x (s, R_s, U_s)}{p(s,R_s,U_s)} + b(s,R_s)\right)^2\, ds\right) \\
 & \leq & C \left( 1+ \nu_T^{-2}+ \nu_T^{-2} \int_0^t \bbE\left[R_s^2\right] ds\right) \\
 &\leq & C(1+\nu_T^{-2}) +  C \nu_T^{-2}  \int_0^t \bbE\left[R_s^2\right] ds,
\eean
where $C$ is a generic constant. The result then follows from Gronwall's inequality.
\end{proof}\\

In order to be able to use the results of \cite{ko} we first need to  establish the Kushner-Stratonovich equation satisfied by $(\pi_t)_{t \in [0,T]}$. To this end, let $B(\bbR^2) $ denote the set of bounded Borel measurable real valued functions on $\bbR^2$ and consider  the operator $\cA_0: B(\bbR_+ \times \bbR) \mapsto B(\bbR^2)  $ defined by
\[
\cA_0 \phi (t, x)=\frac{\partial \phi}{\partial t}(t,x) +  \half \sigma^2(t) \frac{\partial^2  \phi}{\partial x^2}(t,x) + \sigma^2(t)  b(t,x) \frac{\partial \phi}{\partial x}(t,x),
\]
with the domain $\mathcal{D}(\cA_0)= C^{\infty}_c(\bbR_+ \times \bbR)$, where  $C^{\infty}_c$ is the class of infinitely differentiable functions with compact support.  Due to  Assumptions \ref{AssZ},  \ref{fs} and \ref{bt} imposed on $\sigma$ and $b$,  it is well-known, see e.g. Remark 5.4.17 in \cite{ks}, that the martingale problem for $\cA_0$ is well-posed and its unique solution is given by $(t, U_t)$.  Moreover, the Kushner-Stratonovich equation for the conditional distribution of $U$ is given by the following:
\be \label{ks}
\pi_t f = \pi_0 f + \int_0^t \pi_s( \cA_0 f) ds + \int_0^t \left[ \pi_s(\kappa_s f) -\pi_s \kappa_s \pi_s f \right] dI_s,
\ee
for all $f \in C^{\infty}_c(\bbR)$(see Theorem 4.3.1 in \cite{Be}) Note that $f$ can be easily made an element of $\mathcal{D}(\cA_0)$ by redefining it as $f \mathbf{n}$ where $\mathbf{n} \in  C^{\infty}_c(\bbR_+)$ is such that  $\mathbf{n}(t)=1$ for all $t \in [0,1]$. Thus, the above expression is rigorous. The following theorem is a corollary to Theorem  4.1 in \cite{ko}.
\begin{theorem} \label{munique} Suppose the conditions in  Assumptions \ref{AssZ},  \ref{fs} and \ref{bt} hold. Let $(m_t)$ be an $\cF^X$-adapted \cadlag $\cP$-valued process such that
\be \label{ks1}
m_t f = \pi_0 f + \int_0^t m_s( \cA_0 f) ds + \int_0^t \left[ m_s(\kappa_s f) -m_s \kappa_s m_s f \right] dI^m_s,
\ee
for all $f \in C^{\infty}_c(\bbR)$, where $I^m_t=R_t-\int_0^t m_s\kappa_s\, ds.$ Then, $m_t=\pi_t$ for all $t<T$, a.s..
\end{theorem}
\begin{proof} Proof follows along the same lines as the proof of Theorem 4.1 in \cite{ko}, even though, differently from \cite{ko}, we allow the drift of $R$ to depend on $t$ and $R_t$, too. This is due to the fact that  \cite{ko} used the assumption that the drift depends only on the signal process, $U$, in order to ensure that the joint martingale problem $(R,U)$ is well-posed, i.e. conditions of Proposition 2.2 in \cite{ko} are satisfied.  Note that the relevant martingale problem is well posed in our case  since the system of SDEs in (\ref{sdeR}) has a unique strong solution and the drift and dispersion coefficients are bounded on compact domains over the interval $[0,T]$ (see Proposition 5.3.20 and Remark 5.4.17 in \cite{ks} in this regard).
\end{proof}\\

Now, we can state and prove the following corollary.
\begin{corollary} \label{pdfU} Suppose the conditions in  Assumptions \ref{AssZ},  \ref{fs} and \ref{bt} hold. Then,
\[
\pi_t f = \int_{\bbR} f(z) p(t,R_t,z)\, dz,
\]
for any bounded measurable $f$.
Therefore,
\[
\bbE[f(Z_t)|\cF^X_t]=\int_{\bbR} f(z) \rho(t,X_t,z)\, dz.
\]
\end{corollary}
\begin{proof} We have seen in Lemma \ref{pequivrho}  that $p(t,R_t, \cdot)$ satisfies (\ref{spde0}), i.e., $m_t(dz):=p(t,R_t,z)dz$ solves (\ref{ks1}). Then, it follows from Theorem \ref{munique} that $p(t,R_t, \cdot)$ is the conditional density of $U_t$, which gives the first assertion. The second assertion follows from the explicit relationship between $p$ and $\rho$ as described in Lemma \ref{pequivrho} .
\end{proof}

\subsection{Proof of the main result and application to Ornstein-Uhlenbeck bridges}
Now we have all the results necessary to prove Theorem \ref{mainresult}.

\vspace{0.1in}

\noindent \begin{pf-main} The strong solution, Markov property  and $P^{0,z}(\lim_{t \uparrow 1} X_t= Z_1)=1$ follow from Corollary \ref{xexists}. Moreover, it
follows from Corollary \ref{pdfU}  that
$\rho(t,X_t,z)$ is the conditional density of $Z_t$ given $\cF^X_t$.
Thus, $X$ is a local martingale in its own filtration by Lemma \ref{lmart}
\end{pf-main}\\

Note that using the methods employed in this section one can prove the following theorem as well.
\begin{theorem} Let $Z$ be the unique strong solution on $(\Omega , \cG , (\cG_t) , \bbQ)$  to
\[
Z_t=Z_0 + \int_0^t \sigma(s)d\beta_s + \int_0^t \sigma^2(s) b(s, Z_s)ds,
\]
where $b \in C_b^{1,2}$ with bounded derivatives and $\sigma$ is as before.  Suppose $P(Z_0 \in dz) = \Gamma(0,0; c, z) dz$ for some $c \in (0,1)$. Let $\rho(t,x,z):=\Gamma(t,x;V(t),z)$ where $V$ is as defined earlier and $\Gamma(t,x;u,z)$ is the fundamental solution  of (\ref{eq:pde1}). Define $X$ by
\[
dX_t= dB_t + \left\{b(s,X_s)+ \frac{\rho_x(t, X_t, Z_t)}{\rho(t,X_t,Z_t)}\right\}dt,
\]
for $ t \in (0,1)$ with $X_0=0$. Then
\begin{enumerate}
\item In the filtration generated by $X$
\[
X_t-\int_0^t b(s,X_s)ds
\]
defines a standard Brownian motion;
\item $X_1=Z_1$, $P^{0,z}$-a.s. where $P^{0,z}$ is the law of  $(X,Z)$ with $Z_0=z$
and $X_0=0$.
\end{enumerate}
\end{theorem}

The following example show that boundedness of $b$ in the above theorem is not a necessary condition for the result to hold as long as $b$ depends linearly on $x$.
\begin{example} Suppose $Z$ is an Ornstein-Uhlenbeck type process, i.e.
\[
dZ_t= \sigma(t) d\beta_t - k \sigma^2(t) Z_t dt,
\]
where $k>0$ is a constant. Note that in this case the fundamental solution  of (\ref{eq:pde1}) is given by
\[
\Gamma(s,x;t,z)=q((1-e^{-2 k (t-s)})/ 2 k, x e^{-k (t-s)}, z).
\]
Let $X$ be defined by $X_0=0$ and
\[
dX_t= dB_t +\left\{2 k \frac{Z_t- X_t e^{-k (V(t)-t)}}{e^{k (V(t)-t)}-e^{- k (V(t)-t)}}-k X_t \right\}dt,
\]
for $t \in (0,1)$. Then, we claim that if $Z_0$ has a probability density given by $\Gamma(0,0;c,\cdot)$ then $X$ is an Ornstein-Uhlenbeck process in its own filtration and $X_1=Z_1, P^{0,z}$-a.s. under an appropriate modification of Assumption \ref{AssZ}.\ref{tomodify}, which we will state later. Using the method employed in the proof of Proposition \ref{c3:g} we have that $\Gamma(t,x;V(t),z)$ is the conditional density of $Z_t$ given $\cF^X_t$ for $t<1$. Thus, it remains to show that $X_1=Z_1$, $P^{0,z}$-a.s..

Note that the above convergence will be obtained if one can find a continuous function $b(t)$ with $b(1)=1$ such that $X_t- b(t) Z_t$ converges to $0$ as $t \uparrow 1$. We will choose this new function so that that $Y_t:=X_t- b(t) Z_t$ defines a Markov process. It can be checked directly that if $b$ satisfies the following ordinary differential equation
\be \label{exb}
b'(t)+\gamma(t)  b(t) = \theta(t),
\ee
where
\bean
\gamma(t)&=& \frac{e^{-c (V(t)-t)}+e^{c (V(t)-t)}}{e^{c (V(t)-t)}-e^{-c (V(t)-t)}}+c \sigma^2(t), \qquad \mbox{and} \\
\theta(t)&=& \frac{2 c}{e^{c (V(t)-t)}-e^{-c (V(t)-t)}},
\eean
then $Y$ satisfies the following SDE
\be \label{exY}
dY_t= dB_t-b(t)\sigma(t)d\beta_t - c \frac{e^{-c (V(t)-t)}+e^{c (V(t)-t)}}{e^{c (V(t)-t)}-e^{-c (V(t)-t)}} Y_t dt.
\ee
The solution to (\ref{exb}) with the boundary condition $b(1)=1$ is given by
\[
b(t)=\frac{\int_0^t e^{\int_0^s \gamma(r) dr} \theta(s) ds }{e^{\int_0^t \gamma(r) dr}}.
\]
In order to show $Y_t$ converges to $0$ as $t \uparrow 1$ consider the function $\varphi$ defined by
\[
\varphi(t,y):=\frac{1}{\sqrt{2(\Lambda(t) + \ell)}}e^{\frac{y^2}{2 \lambda^2(t)(\Lambda(t)+\ell)}},
\]
where
\bean
\lambda(t)&:=&\exp\left({-c \int_0^t  \frac{e^{-c (V(s)-s)}+e^{c (V(s)-s)}}{e^{c (V(s)-s)}-e^{-c (V(s)-s)}} ds}\right), \qquad \mbox{and} \\
\Lambda(t)&:=&\int_0^t \frac{1+ b^2(s) \sigma^2(s)}{\lambda^2(s)} ds.
\eean
A direct application of It\^{o}'s formula gives that $\varphi(t,Y_t)$ is a positive local martingale, hence a super-martingale. If  Assumption \ref{AssZ} holds with $\lambda$ and $\Lambda$ defined above, then we can imitate the proof of  Proposition \ref{bpconv} using $\varphi(t,y)$ defined above to conclude that $Y_t$ converges to $0$ as $t \uparrow 1$.
\end{example}

\section{Application to finance: A generalization of Back-Pedersen equilibrium model}\label{application}

We will use the previous bridge construction to solve an equilibrium
model with information asymmetry that can be viewed as a
non-Gaussian generalization of Back and Pedersen's \cite{BP}. We keep the notation of the previous sections, in particular all the stochastic processes will be defined on $(\Omega , \cG , (\cG_t) , \bbQ)$.

Consider a stock issued by a company with fundamental value given by the diffusion  process $\mathcal{Z}=(\Omega , \cG , (\cG_t) , (Z_t) ,
(P^z)_{z\in \bbR})$ with values in
$\bbR$, and satisfying
\begin{equation}
Z_t=Z_0+\int_0^t\sigma(s)a(V(s),Z_s)d\beta_s
\end{equation}
where $\beta$ is a standard Brownian motion adapted to $(\cG_t)$, $a$ and $\sigma$ are deterministic functions, $V(t)=c+\int_0^t\sigma^2(s)ds$, for some constant $c$ and the probability density of $Z_0$ is  $G(0, 0; c, \cdot)$ with $G$ given by (\ref{GviaGamma}). We will require $a$ and $\sigma$ satisfy some further assumptions which are  made precise in Assumption \ref{a:model} below.

Then, if the firm value is observable, the fair stock price should be a
function of $Z_{t}$ and $t$. However, the assumption of the company value being
discernible by the whole market in continuous time is counter-factual, and it
will be more realistic to assume that this information is revealed to the
market only at given time intervals (such as dividend payments times or when
balance sheets are publicized).

In this model we therefore assume, without loss of generality, that
the time of the next information release is $t=1$, and the market
terminates after that. Hence, in this setting the stock can be viewed as
a European option on the firm value with maturity  $T=1$ and payoff
$f(Z_{1})$. In addition to this risky asset, there is a riskless
asset that yields an interest rate normalized to zero for
simplicity of exposition.

The microstructure of the market, and the interaction of market
participants, is modeled as a generalization of \cite{BP}. There are
three types of agents: noisy/liquidity traders, an informed
trader (insider), and a market maker, all of whom are risk
neutral. The agents differ in their information sets, and objectives, as
follows.

\begin{itemize}
\item \textit{Noisy/liquidity traders} trade for liquidity reasons, and
their total demand at time $t$ is given by a standard $(\cG_t)$-Brownian motion $
B$ independent of $\beta$ and $Z_0$.

\item \textit{Market maker} observes only the total market order process $
Y_{t}=\theta _{t}+B_{t}$, where $\theta _{t}$ is the total order of the
insider at time $t$ which is an absolutely continuous process, and therefore $Y$ is a continuous semimartingale on $(\Om, \cG, (\cG_t), \bbQ)$.  This in particular implies that the market maker's  filtration is $\mathcal{F}_{t}^{Y}$.  Similar to \cite{Cho}, we assume that the market maker sets the price as a
function of weighted total order process at time $t$, i.e. we
consider pricing functionals $
S\left( Y_{[0,t]},t\right) $ of the following form%
\begin{equation} \label{eq:rule_mm}
S\left( Y_{[0,t]},t\right) =H\left(t, X_t\right), \qquad \forall t\in [0,1)
\end{equation}
where $X$ is the unique strong solution of
\begin{equation}\label{eq:signal_mm}
dX_t = w (t,X_t) dY_t,\quad \forall t\in [0,1), \, X_0 =0
\end{equation}
on $(\Om, \cG, (\cG_t), \bbQ)$ for some deterministic function $w(s,x)$ chosen by the market maker.  Moreover, a pricing rule $(H,w)$ has to be admissible in the sense of Definition \ref{def:prule}.  In particular, $H \in C^{1,2}$ and, therefore, $S$ is a semimartingale on $(\Om, \cG, (\cG_t), \bbQ)$ on $[0,1)$.
\item \textit{The informed investor} observes the price process $
S_{t}=H\left(t, X_t\right)$ where $X$ is given by (\ref{eq:signal_mm}), and the true firm value $Z_{t}$,
i.e. her filtration is given by
$(\mathcal{F}_{t}^{Z,S})$. Since she is
risk-neutral, her objective is to maximize the expected final
wealth, i.e.
\begin{equation}
\sup_{\theta \in \mathcal{A}(H,w)}E^{0,z}\left[ W_{1}^{\theta
}\right] =\sup_{\theta \in \mathcal{A}(H,w)}E^{0,z}\left[
(f(Z_{1})-S_{1-})\theta _{1}+\int_{0}^{1-}\theta _{s}dS_{s}\right]
\label{ins_obj}
\end{equation}%
where $E^{0,z}$ is the expectation with respect to the probability measure $P^{0,z}$ which is the law of $(X,Z)$ with $X_0=0$ and $Z_0=z$, and $\mathcal{A}(H,w)$ is the set of admissible trading strategies
for the given pricing rule $(H,w)$, which will be defined in Definition \ref{admissible}. That is, the insider maximizes the
expected value of her final wealth
$W_{1}^{\theta }$, where the first term on the right hand side of equation (%
\ref{ins_obj}) is the contribution to the final wealth due to a potential
differential between price and fundamental at the time of information
release, and the second term is the contribution to final wealth coming from
the trading activity.
\end{itemize}
\begin{remark} Note that by setting $\sigma \equiv 0$ and $c=1$, we obtain the ``static information market'' considered by \cite{B}. Moreover, setting $a\equiv 1$ results in the model studied by \cite{BP}.

In both cases, $V(t)-t$ was a measure of the uncertainty of the market about the value of $Z_t$ which is equivalent to the informational advantage of the insider in comparison with the market maker (see discussion at the beginning of p.~393 of \cite{BP}). As we will see later in Remark \ref{r:ia}, this observation remains valid in our generalized case.
\end{remark}

\begin{remark} We stress the fact that the total demand $Y=Y^\theta$ depends on insider's strategy $\theta$, so that the market maker's filtration $\mathcal F^Y = \mathcal F^{Y^\theta}$ depends also on $\theta$ (through $Y$). To avoid heavy notation, we will drop the superscript $\theta$ from total demand. So will simply write $Y$ and $\mathcal F^Y$ instead of $Y^\theta$ and $\mathcal F^{Y^\theta}$, respectively.
\end{remark}

Note also that the above market structure implies that the insider's
optimal trading strategy takes into account the \emph{feedback
effect }i.e. that prices react to her trading strategy according
to (\ref{eq:rule_mm}) and (\ref{eq:signal_mm}). Our goal is to find the rational
expectations equilibrium of this market, i.e. a pair consisting of
an \emph{admissible} pricing rule and an \emph{admissible}
trading strategy such that: \textit{a)} given the pricing rule
the trading strategy is optimal, \textit{b)} given the trading
strategy, there exists a unique strong solution, $X_t$, of (\ref{eq:signal_mm}) over the time interval $[0,1)$, and the pricing rule is {\em rational} in the following sense:
\be \label{mm_obj}
H(t,X_t)=S_t=\mathbb{E}^{\bbQ}\left[f(Z_1)|\mathcal{F}_t^Y\right]
\ee
with $S_1=f(Z_1)$.
To formalize
this definition of equilibrium, we first need to define the sets of admissible
pricing rules and trading strategies.

The definition of admissible pricing rules is a generalization
 of the one in \cite{B} and \cite{BP}. This generalization
allows the market maker to re-weight his past information with a weighting function $w$ as stated in (\ref{eq:signal_mm}).

\begin{definition}\label{def:prule} For a given  semimartingale $Y$ on $(\Om, \cG, (\cG_t), \bbQ)$, an {\em admissible
pricing rule} is any pair $(H,w)$ fulfilling the following
conditions:
\begin{enumerate}
\item $w : [0,1]\times \bbR \mapsto \bbR_+$ is a function in $C^{1,2}([0,1] \times \bbR)$  bounded away from $0$.
\item There exists a unique strong solution of
\begin{equation}\label{eq:signal_mm1}
dX_t = w (t,X_t) dY_t,\quad X_0 =0
\end{equation}
over the time interval $[0,1)$ on $(\Om, \cG, (\cF^Y_t), \bbQ)$;
\item $H \in C^{1,2}([0,1] \times \bbR)$;
\item $x \mapsto H (t,x)$ is strictly increasing for every $t\in [0,1]$;
\end{enumerate}
Moreover, given $\theta \in \cA(H,w)$, a pricing rule $(H,w)$ is said to be \emph{rational} if it satisfies (\ref{mm_obj}).
\end{definition}
\begin{remark} The strict monotonicity of $H$ in the space variable implies $H$ is invertible,
thus, the filtration of the insider is generated by $X$ and $Z$. Moreover, since $w$ is bounded away from $0$ the filtrations generated by $X$ and $Y$ are the same. This in turn implies that $(\cF^{S,Z}_t)=(\cF^{B,Z}_t)$, i.e. the insider has full information about the market.
\end{remark}

It is standard (see, e.g., \cite{BP}, \cite{Cho} or \cite{Wu}) in the insider trading literature to limit
the set of admissible strategies to absolutely continuous ones motivated by the result in Back \cite{B}, and we do so. The formal definition of the set
of admissible trading strategies is summarized in the following
definition.

\begin{definition} \label{admissible}
An $\cF^{B,Z}$-adapted  $\theta$ is said to be an  admissible trading
strategy for a  given pair $(H,w)$  if
\begin{enumerate} \item it is  absolutely continuous with respect to the Lebesgue
measure, i.e. $\theta_t = \int_0 ^t \alpha_s ds$;
\item There exists a unique strong solution, $X$, to the SDE\footnote{Note that this SDE is well defined on $(\Om, \cG, (\cG_t), \bbQ)$ since, due to absolute continuity of $\theta$, $Y=B+\theta$ is a semimartingale on it.} (\ref{eq:signal_mm1}) on $(\Om, \cG, (\cF^{B,Z}_t), \bbQ)$ over the interval $[0,1)$.
\item $(X,Z)$ is a Markov process adapted to $(\cG_t)$ with law $P^{0,z}$;
\item and no doubling strategies are
allowed i.e.
\begin{equation}
E^{0,z}\left[ \int_{0}^{1}H^{2}\left(t,X_t\right)dt\right] <\infty.
\label{theta_cond_2}
\end{equation}
 The set of admissible trading strategies for the  given pair $(H,w)$ is denoted with $\mathcal{A}(H,w)$.
\end{enumerate}
\end{definition}

Given these definitions of admissible pricing rules and trading strategies,
it is now possible to formally define the market equilibrium as follows.
\begin{definition} \label{eqd} A triplet $(H^{\ast},w^{\ast}, \theta^{\ast})$
is said to form an equilibrium if $(H^{\ast},w^{\ast})$ is an admissible pricing rule for the semimartingale $Y^{\ast}=B+\theta^{\ast}$,
$\theta^{\ast} \in \cA(H^{\ast},w^{\ast})$, and the following conditions are
satisfied:
\begin{enumerate}
\item {\em Market efficiency condition:} given $\theta^{\ast}$,
$(H^{\ast},w^{\ast})$ is a rational pricing rule. \item {\em Insider
optimality condition:} given $(H^{\ast},w^{\ast})$, $\theta^{\ast}$ solves
the insider optimization problem:
\[
E^{0,z}[W^{\theta^{\ast}}_1] = \sup_{\theta \in \cA(H^{\ast},w^{\ast})} E^{0,z} [W^{\theta}_1].
\]
\end{enumerate}
\end{definition}
Additionally, to define a well behaved problem we impose the following
technical conditions on the model parameters.

\begin{assumption}\label{assf} $f:\bbR \mapsto \bbR$ is a strictly increasing function belonging to $C^1$ such that
\[
|f(z)| \leq k_1 \exp\left(k_2 A(1,z)\right), \qquad \forall z \in \bbR
\]
for some constants $k_1$ and $k_2$ where $A$ is given by (\ref{ef}).
\end{assumption}
\begin{remark} The assumption that $f$ is
strictly increasing implies that the larger the signal $Z$ the
larger the value of the risky asset for the insider. This assumption
will also play a role in order to prove that the proposed
equilibrium pricing rule satisfies condition 4 in Definition
\ref{def:prule}. \end{remark}
\begin{assumption}\label{a:model}
We assume that the parameters of the model satisfy the following assumptions:
\begin{enumerate}
\item Both $a(t,z)$ and $\sigma(t)$ satisfy Assumption \ref{AssZ}.
\item $a(t,z)$ also satisfies a nonlinear PDE:
\begin{equation}\label{PDEa}
a_t(t,z)+\frac{a^2(t,z)}{2}a_{zz}(t,z)=0
\end{equation}
\end{enumerate}
\end{assumption}

\begin{remark} \label{rem:AG} Due to definition of $b$ given in (\ref{bee}) we have
\bean
 b(t,x)&=&A_t(t,A^{-1}(t,x))-\half a_z(t,A^{-1}(t,x))\\
&=&-\int_0^{A^{-1}(t,x)}\frac{a_t(t,y)}{a^2(t,y)}\,dy - \half a_z(t,A^{-1}(t,x))\\
&=&\half \int_0^{A^{-1}(t,x)}a_{zz}(t,y)\,dy - \half a_z(t,A^{-1}(t,x))\\
&=&-\half a_z(t,0),
\eean
where the second equality is due to the definition of $A$ (see \ref{ef}) and the third equality follows from Assumption \ref{a:model}.2. Therefore, $b$ is continuous and depends only on $t$.
Moreover, Assumptions  \ref{fs} and \ref{bt} are automatically satisfied since $a \in C^{1,2}$ and $t \in[0,1]$. In this case $U_t=A(V(t), Z_t)$, where $A$ is defined in (\ref{ef}), is a Gaussian process.  Moreover, in the next subsection we will give some heuristics indicating that $a(t,z)$ being a solution to (\ref{PDEa}) is a necessary condition for the existence of an equilibrium. This suggests the conclusion that the only possible form for the signal $Z$, for which an equilibrium in the sense of Definition \ref{eqd} exists, is $Z_t=\Phi(t, U_t)$ where $U$ is a Gaussian process and $\Phi$ is a deterministic function.
\end{remark}

\subsection{Equilibrium}

First, we shall provide some heuristics in order to motivate the PDE
(\ref{PDEa}) we imposed on $a(t,z)$.

Let $(H,w)$ be any rational pricing rule. First, notice that a standard application of
integration-by-parts formula applied to $W_1 ^\theta$ for any
$\theta \in \cA(H,w)$  gives
\be \label{eq:w2} W_1 ^\theta = \int_0^1 (f(Z_1)-S_s) \alpha_s \,ds .\ee Furthermore,
\be \label{eq:cew2}
E^{0,z}\left[\int_0 ^1 (f(Z_1)-S_s) \alpha_s ds \right]=E^{0,z}\left[\int_0 ^1 (E^{0,z}[f(Z_1)|\cF^{B,Z}_s]-S_s) \alpha_s ds \right].
\ee
Define the value, $P$, of the stock for the insider by
\be \label{d:P}
P_t:= E^{0,z} [f(Z_1)|\cF_t^{B,Z}]=E^{0,z} [f(Z_1)|\cF_t ^Z] = F(t,Z_t),
\ee
for some measurable function $F:[0,1]\times \bbR \mapsto \bbR$ (due to independence between $Z$ and $B$ and the Markov
property of $Z$). Note that this expectation is well defined since, due to Assumption \ref{assf}, $|f(Z_1)|\leq k_1\exp(k_2 U_1)$ where $U_1$ is a Gaussian random variable. Moreover, $P_1=f(Z_1), P^{0,z}$-a.s. for every $z \in \bbR$, and the function $F$ is given by
\be \label{d:f}
F(t,z)=\int_{\bbR} f(y) G(V(t),z;1,y)\,dy,
\ee
where $G$ is the function defined in Proposition \ref{Gexists}.
 Due to Assumption \ref{assf} on $f$, it follows from Theorem 12 in Chap. I of \cite{friedman} that  $F \in C^{1,2}([0,1] \times \bbR)$
and satisfies \be \label{pdeF} F_t(t,z) + \half \sigma^2(t)
a^2(V(t),z)F_{zz}(t,z)=0. \ee
In view of (\ref{eq:w2}) and (\ref{eq:cew2}), insider's optimization problem becomes
\begin{equation}\label{OptInsider}
\sup_{\theta \in \cA(H,w)} E^{0,z} [ W_1 ^\theta ] = \sup_{\theta \in \cA(H,w)}
E^{0,z}\left[\int_0 ^1 (F(s,Z_s) - H(s, X_s)) \alpha_s ds
\right].\end{equation}
Recall that the signal $Z_t$ follows
\[
dZ_t=\sigma(t)a(V(t),Z_t)d\beta_t.
\]
Suppose that $\theta_t=\int_0^t\alpha(s,X_s,Z_s)ds$ is a solution of
the problem (\ref{OptInsider}). Then the market price is given by
$H(t,X_t)$ with
\[
dX_t =w(t,X_t)\alpha(t,X_t,Z_t)dt + w(t,X_t)dB_t
\]
Let \[
 J(t,x,z) := \esssup_{\theta \in \cA(H,w)} E^{0,z} \left[ \int_t ^1
(F(s,Z_s)-H(s,X_s))d\theta_s | X_t =x, Z_t = z \right], \quad t\in
[0,1] \] be the associated value function of the insider's problem.
Applying formally the dynamic programming principle, we get the following HJB equation:
\begin{equation}
0=\sup_{\alpha}\left(\left[w(t,x) J_x
+F(t,z)-H(t,x)\right]\alpha \right)+J_t+\frac{1}{2} w^2(t,x)
J_{xx}+\frac{1}{2}\sigma^2(t)a^2(V(t),z)J_{zz}
\end{equation}
Thus, for the existence of an optimal $\alpha$ we need
\begin{eqnarray}\label{eq:wf_1}
w(t,x) J_x+F(t,z)-H(t,x)=0\\ \label{eq:wf_2} J_t+\frac{1}{2}w^2(t,x)
J_{xx}+\frac{1}{2}\sigma^2(t)a^2(V(t),z)J_{zz}=0
\end{eqnarray}
Differentiating (\ref{eq:wf_1}) with respect to $x$
and since from (\ref{eq:wf_1}) it follows that
$J_x=\frac{H(t,x)-F(t,z)}{w(t,x)}$, we get \be \label{eq:wf_4}
w^2(t,x) J_{xx}=H_x(t,x)w(t,x)+(F(t,z)-H(t,x))w_x(t,x) \ee Plugging
(\ref{eq:wf_4}) into (\ref{eq:wf_2}) yields: \be \label{eq:wf_5}
J_t+\frac{1}{2}\left(H_x(t,x)w(t,x)+(F(t,z)-H(t,x))w_x(t,x)\right)+\frac{1}{2}\sigma^2(t)a^2(V(t),z)J_{zz}=0
\ee Differentiating (\ref{eq:wf_1}) with respect to $z$ gives
$J_{xz}=-\frac{F_z(t,z)}{w(t,x)}$ and therefore
$J_{zzx}=-\frac{F_{zz}(t,z)}{w(t,x)}$. Thus, after differentiating
(\ref{eq:wf_5}) with respect to $x$ we obtain: \be \label{eq:wf_6}
J_{tx}+\frac{1}{2}\left(H_{xx}(t,x)w(t,x)+(F(t,z)-H(t,x))w_{xx}(t,x)\right)-\sigma^2(t)
\frac{a^2(V(t),z)}{2w(t,x)}F_{zz}(t,z)=0 \ee Since differentiation
(\ref{eq:wf_1}) with respect to $t$ gives
$$J_{xt}=\frac{w_t(t,x)}{w^2(t,x)}(F(t,z)-H(t,x))-\frac{1}{w(t,x)}(F_t(t,z)-H_t(t,x)),$$
(\ref{eq:wf_6}), in view of (\ref{pdeF}), implies \be \label{eq:wf_7} (H(t,x)
-F(t,z))\left\{w_t(t,x)+\frac{w^2(t,x)}{2}w_{xx}(t,x)\right\}=w(t,x)\left(H_t(t,x)+\half
w^2(t,x)H_{xx}(t,x)\right). \ee  Since the
right hand side of (\ref{eq:wf_7}) is not a function of $z$, we
must have
\bea w_t(t,x)+\frac{w^2(t,x)}{2}w_{xx}(t,x)&=&0, \label{eq:pdew} \\
H_t(t,x)+\half w^2(t,x)H_{xx}(t,x)&=&0. \label{pdeh} \eea

\begin{remark} In Proposition \ref{SuffCondIns} we show that if the system of PDEs given by (\ref{eq:pdew}) and (\ref{pdeh}) are satisfied, then there exists an optimal strategy for the insider.
Under further assumptions one can show that the requirement on $(H,w)$ posed by the PDEs (\ref{eq:pdew}) and (\ref{pdeh}) is in fact a necessary condition for the existence of an optimal solution for the insider. Indeed, if $f$ is bounded, and therefore, $F$ and $H$ are bounded, and there exists an optimal strategy for the insider such that the value function is in $C^{1,2,2}$, then  Theorem 4.3.1 in \cite{Pham} gives that $J$ has to satisfy simultaneously (\ref{eq:wf_1}) and (\ref{eq:wf_2}).   Thus, $w$
 satisfying the nonlinear PDE above is a necessary condition in
 order to have a smooth value function $J$.
\end{remark}

An examination of (\ref{pdeh}) suggests that $X$ associated with the optimal strategy is a martingale in its own
 filtration. In view of these observations,
 recalling the bridge construction in the previous section with certain properties, it is easily seen that $a(t,x)$
 is a natural candidate for the equilibrium weight function $w^{\ast}(t,x)$.
That is the  reason why we need to assume that $a$ satisfies PDE
(\ref{PDEa}).

\begin{remark} PDE (\ref{PDEa}) admits many explicit solutions satisfying the
properties listed in Assumption \ref{AssZ}. 
Here are few examples taken from \cite{PZ}, sections from 1.1.9.10 to 1.1.9.13 and from 1.1.9.18 to 1.1.9.20.
\begin{enumerate}
\item[(i)] $a(t,z)=a_0$ for some constant $a_0 >0$, which is the case already studied by Back and Pedersen \cite{BP};
\item[(ii)] $a(t,z)=\sqrt{k_1 (z+k_2)^2 +k_3 e^{-k_1 t} }$, where $k_1,k_2,k_3$ are positive constants. Indeed, since $t$ varies on $[0,1]$, $\inf_z a(t,z) \geq \sqrt{k_3}e^{-2k_1}$, so that $a(t,z)$ is uniformly bounded away from zero.
\item[(iii)] $a(t,z)=\frac{g(z)}{\sqrt{k_1 t + k_2}}$ where $g$ is solution to $\frac{k_1}{g}=g^{\prime\prime}$ and is bounded away from zero.
\item[(iv)] (Self-similar solution) $a(t,z)=y(z/\sqrt{t})$, where $y(x)$ satisfies $y^2 y_{xx} - y_x x = 0$  and is bounded away from $0$.
\item[(v)] (Generalized self-similar solution) $a(t,z)=e^{-2k_1 t}\, y(ze^{2k_1t})$, where $y(x)$ satisfies
\[-\frac{1}{2} y^2 y_{xx} = 2k_1 x y_x -2k_1 y\]
and is bounded away from $0$.
\end{enumerate}
\end{remark}

The next proposition describes the optimal insider's strategy in terms of the behavior of the resulting optimal demand at maturity.

\begin{proposition}\label{SuffCondIns}
Assume that $(H,w)$ satisfy
\begin{equation}\label{eq:opt_str_pde_pr_func_1}
H_t(t,x)+\frac{w(t,x)^2}{2}H_{xx}(t,x)=0
\end{equation}
and
\begin{equation}\label{eq:opt_str_pde_weight_func}
w_t(t,x)+\frac{w(t,x)^2}{2}w_{xx}(t,x)=0.
\end{equation}
If $\theta^{\ast} \in \cA(H,w)$ satisfies $S_1 := H(1, X^{\ast} _1) = F(1,Z_1), P^{0,z}$-a.s. for every $z \in \bbR$, where $X^{\ast}$ is the solution to $X_t = \int_0 ^t w (s,X_s) dY_s ^{\ast}$ with $Y^{\ast} = B +\theta^{\ast}$, and $(H,w)$ is admissible for $Y^{\ast}$, then $\theta^{\ast}$
is an optimal strategy, i.e.,
\[ E^{0,z} [W_1 ^{\theta^{\ast}}] \geq E^{0,z} [W_1 ^\theta ]\]
a.s. for all $\theta \in \cA(H,w)$.
\end{proposition}

\begin{proof} We will adapt Wu's proof of his Lemma 4.2 in \cite{Wu}. Consider the function
 \begin{equation}\label{def:generalG_a}
\Psi^a(t,x):=\int_{\xi(t,a)} ^x \frac{H(t,u)-a}{w(t,u)}du+\frac{1}{2}\int_t^1H_x(s,\xi(s,a))w(s,\xi(s,a))ds
\end{equation}
where $\xi(t,a)$ is the unique solution of $H(t,\xi(t,a))=a$.
Direct differentiation with respect to $x$ gives that
\begin{equation}\label{eq:G_a_x}
\Psi^a_x(t,x)w(t,x)=H(t,x)-a.
\end{equation}
Differentiating above with respect to $x$ gives
\begin{equation}\label{eq:G_a_xx}
\Psi^a_{xx}(t,x)w^2(t,x)=w(t,x)H_x(t,x)-\left(H(t,x)-a\right)w_x(t,x).
\end{equation}
Direct differentiation of $\Psi^a(t,x)$ with respect to $t$ gives
\begin{eqnarray}
\Psi^a_t(t,x) &=&\int_{\xi(t,a)} ^x \frac{H_t(t,u)}{w(t,u)}du-\int_{\xi(t,a)} ^x \frac{(H(t,u)-a)w_t(t,u)}{w^2(t,u)}du-\frac{1}{2}H_x(t,\xi(t,a))w(t,\xi(t,a)) \nn \\ \nn
           &=&\int_{\xi(t,a)} ^x \frac{H_t(t,u)}{w(t,u)}du+\frac{1}{2}\int_{\xi(t,a)} ^x (H(t,u)-a)dw_x(t,u)-\frac{1}{2}H_x(t,\xi(t,a))w(t,\xi(t,a))\\ \label{eq:G_a_t}
           &=&\frac{1}{2}\left((H(t,x)-a)w_x(t,x)-H_x(t,x)w(t,x)\right)
\end{eqnarray}
where in order to obtain the last equality we used
(\ref{eq:opt_str_pde_pr_func_1}) and integration by parts twice on
the second integral. Combining (\ref{eq:G_a_xx}) and
(\ref{eq:G_a_t}) gives
\begin{equation} \nn
\Psi^a_t+\frac{1}{2}w(t,x)^2\Psi^a_{xx}=0.
\end{equation}
Therefore from (\ref{eq:G_a_x}) and It\^o's formula it follows that,
\begin{equation} \nn
\Psi^a(1,X_1)-\Psi^a(0,X_0)=\int_0 ^1 \frac{H(t,X_t)-a}{w(t,X_t)} dX_t,
\end{equation}
and in particular, when $a=F(1,Z_1)$,
\begin{equation} \label{eq:opt_str_G_a}
\Psi^{F(1,Z_1)}(1,X_1)-\Psi^{F(1,Z_1)}(0,X_0)=\int_0 ^1 \frac{H(t,X_t)-F(1,Z_1)}{w(t,X_t)} dX_t.
\end{equation}
Using (\ref{eq:w2}), (\ref{eq:opt_str_G_a}) and admissibility properties of $\theta$, in particular $d\theta_t=\alpha_t dt$, the insider optimization problem becomes
\begin{eqnarray}
\sup_{\theta\in {\cal{A}}(H,w)}E^{0,z} [W^{\theta}_1]&=& \sup_{\theta\in {\cal{A}}(H,w)}E^{0,z} \left[\int_0 ^1 \left(F(1,Z_1)- H(t,X_t)\right) d\theta_t \right]\\
&=& E^{0,z} \left[\Psi^{F(1,Z_1)}(0,X_0)\right]-\inf_{\theta\in {\cal{A}}(H,w)}E^{0,z}\left[\Psi^{F(1,Z_1)}(1,X_1) \right]
\end{eqnarray}
where the last equality is due to (\ref{theta_cond_2}) in Definition \ref{admissible}, and
\[ E^{0,z} \left[\left(\int_0 ^1 F(1,Z_1)dB_t \right)^2 \right] = E^{0,z} \left[ F(1,Z_1)^2 \right]E^0[B^2 _1] < \infty,\]
since $Z$ and $B$ are independent.

The conclusion follows from the fact that $\Psi^{F(1,Z_1)}(1,X_1)=\int_{\xi(1,F(1,Z_1))} ^{X_1} \frac{H(1,u)-F(1,Z_1)}{w(t,u)}du$ which, due to the fact that $H(t,x)$ is increasing and $w(t,u)$ is positive, is positive unless $X_1=\xi(1,F(1,Z_1))$, that is, $H(1,X_1)=F(1,Z_1)$. Therefore, an insider trading strategy which gives $H(1,X_1)=F(1,Z_1)$ is optimal.
\end{proof}\\

We have the following sufficient condition for a triplet $(H^{\ast},w^{\ast},\theta^{\ast})$ to be an equilibrium.

\begin{lemma} \label{SuffEqui} A triplet $(H^{\ast},w^{\ast} ,\theta^{\ast})$ where $(H^{\ast},w^{\ast})$ is an admissible pricing rule for the semimartingale $Y^{\ast}=B+\theta^{\ast}$, and $\theta^{\ast} \in \cA(H^{\ast},w^{\ast})$, is an equilibrium if it fulfills the following four conditions
\begin{enumerate}
\item \label{pdeH} $H^{\ast}(t,x)$ satisfies the PDE $H^{\ast} _t (t,x) + \frac{1}{2} w^{\ast}(t,x)^2 H^{\ast} _{xx}(t,x)=0$ for any $(t,x)\in [0,1)\times \bbR$.
\item \label{pdew} Weighting function satisfies $w^{\ast}_t(t,x)+\frac{w^{\ast}(t,x)^2}{2}w^{\ast}_{xx}(t,x)=0$.
\item \label{optY} $Y_t ^{\ast} = B_t +\theta_t ^{\ast} $ is a standard BM in its own filtration.
\item \label{bridge} $H^{\ast}(1,X_1 ^{\ast}) = f(Z_1), P^{0,z}$-a.s. for every $z \in \bbR$ where $X ^{\ast}$ is the solution to $X_t = \int_0 ^t w (s,X_s) dY_s ^{\ast}$ with $Y^{\ast} = B +\theta^{\ast}$.
\item \label{hmart} $(H^{\ast}(t,X^{\ast}_t))_{t \in [0,1]}$ is an $(\cF^{Y^{\ast}}_t)$-martingale with respect to $\bbQ$.
\end{enumerate}
\end{lemma}

\begin{proof} Let $(H^{\ast},w^{\ast},\theta^{\ast})$ be a triplet satisfying conditions 1 to 4 above. By Proposition \ref{SuffCondIns}, conditions 1,2 and 4 imply that $\theta^{\ast}$ is optimal. On the other hand, 1, 3, 4 and 5 imply that the pricing rule $(H^{\ast},w^{\ast})$ is rational. \end{proof}\\

Combining Proposition \ref{SuffEqui} and the bridge construction given in the previous section, we can finally state and prove the main result of this section. We recall from Proposition \ref{Gexists} that the function $G=G(t,x;u,y)$ is the transition density of
\be \label{xii}
d\xi_t=a(t,\xi_t)\,d\beta_t,
\ee
and from Theorem \ref{mainresult} that there exists a unique strong solution under $\mathcal F^{B,Z}$ of the following SDE:
\[dX_t = a(t,X_t)dB_t + a^2(t,X_t) \frac{\rho_x (t,X_t,Z_t)}{\rho(t,X_t,Z_t)} dt,\quad X_0 =0. \]
\begin{theorem} \label{t31}
Under Assumptions \ref{assf} and \ref{a:model}  there exists an equilibrium $(H^{\ast},w^{\ast},\theta^{\ast})$, where
\begin{enumerate}
\item[(i)] \label{optHw} $H^{\ast}(t, x) =F(V^{-1}(t),x)$ where $F$ is given by (\ref{d:f}) and $w^{\ast}(t,x) = a(t,x)$ for all $(t,x)\in [0,1]\times \bbR$;
\item[(ii)] $\theta_t ^{\ast} = \int_0 ^t \alpha^{\ast} _s ds$ where $\alpha^{\ast} _s = a(s,X_s) \frac{\rho_x (s,X_s,Z_s)}{\rho(s,X_s,Z_s)}$ with $\rho(t,x,z)=G(t,x;V(t),z)$ and the process $X^{\ast}$ is the unique strong solution under $\mathcal F^{B,Z}$ of the following SDE:
\[dX_t = a(t,X_t)dB_t + a^2(t,X_t) \frac{\rho_x (t,X_t,Z_t)}{\rho(t,X_t,Z_t)} dt,\quad X_0 =0. \]
\end{enumerate}
\end{theorem}

\begin{proof}
We will first show that $(H^{\ast},w^{\ast})$ is admissible in the sense of Definition \ref{def:prule}. Note that since $F \in C^{1,2}$,  so is $H^{\ast}\in C^{1,2}([0,1]\times \bbR)$ and $w^{\ast}$ is bounded away from $0$ since $a(t,z)$ is assumed to be bounded away from $0$ in Assumption \ref{a:model}. We also have that $X^*$ is the unique strong solution to
\[
X_t=\int_0^t w^{\ast}(s,X_s)dY^{\ast}_s,
\]
on $(\Om, \cG, (\cF^{Y^{\ast}}_t), \bbQ)$ by condition ii) of the theorem and that $dX_t= a(t,X_t)dY^{\ast}_t= w^{\ast}(t,X_t)dY^{\ast}_t$.
In order to complete the proof of admissibility we next show that $x \mapsto H^{\ast}(t,x)$ is strictly increasing for every $t \in[0,1]$. Observe that this is equivalent to the analogous property for $F$. First, using  (\ref{d:f})  and (\ref{GviaGamma}) we obtain
\bean
F(t,z)&=&\int_{\bbR} f(y) \Gamma(V(t),A(V(t),z);1,A(1,y))\,dA(1,y).\\
&=&\int_{A(1,-\infty)}^{A(1,\infty)} f(A^{-1}(1,y))\Gamma(V(t),A(V(t),z);1,y)\,dy\\
&=&\int_{A(1,-\infty)}^{A(1,\infty)} f(A^{-1}(1,y))q\left(1-V(t),A(V(t),z)+ c(t),y\right)\,dy,
\eean
where the last line follows from Lemma \ref{j:tech} and $q$ is the transition density of standard Brownian motion given by
\[
q(t,x,y)=\frac{1}{\sqrt{2 \pi t}}\exp\left(-\frac{(x-y)^2}{2 t}\right),
\]
and $c(t)=\int_{V(t)}^1 b(s)ds$. Due to bounds on $f$ we can differentiate inside the integral to get that
\bean
F_z(t,z)&=&\int_{A(1,-\infty)}^{A(1,\infty)} f(A^{-1}(1,y))q_x\left(1-V(t),A(V(t),z)+ c(t),y\right)\frac{1}{a(V(t),z)}\,dy\\
&=&-\int_{A(1,-\infty)}^{A(1,\infty)} f(A^{-1}(1,y))q_y\left(1-V(t),A(V(t),z)+ c(t),y\right)\frac{1}{a(V(t),z)}\,dy\\
&=& \int_{A(1,-\infty)}^{A(1,\infty)} f'(A^{-1}(1,y))\frac{a(1, A^{-1}(1,y))}{a(V(t),z)}q\left(1-V(t),A(V(t),z)+ c(t),y\right)\,dy >0,
\eean
where the third equality follows from integration by parts, which is valid due to Assumption \ref{assf}. The final strict inequality is due to the fact that $f$ is strictly increasing and $a$ is strictly positive. Therefore, $(H^{\ast},w^{\ast})$ is admissible for the semimartingale $Y^{\ast}=B+\theta^{\ast}$.

Next, we turn to verify that $\theta^{\ast} \in \cA(H^{\ast},w^{\ast})$. By construction $\theta^{\ast}$ is absolutely continuous. Moreover, the conditions 2 and 3 of Definition \ref{admissible} follow from Theorem \ref{mainresult}. Finally, condition 4 follows from Lemma \ref{j:tech}.

To finish the proof, let us verify that the triplet
$(H^{\ast},w^{\ast},\theta^{\ast})$ given in the statement satisfy the five
conditions of  Lemma \ref{SuffEqui}. First, $H^{\ast}$ as defined  satisfies condition \ref{pdeH} in Lemma \ref{SuffEqui} due to (\ref{pdeF}). The second condition is trivially satisfied due to Assumption \ref{a:model}. For the third condition observe that $X^{\ast}$ is a local martingale in its own filtration due to Theorem \ref{mainresult}. However, since  $a(t,z)$ is uniformly bounded
away from $0$, it follows that the natural filtrations of $X^{\ast}$ and
$Y^{\ast}$ coincide. The conclusion that  $Y^{\ast}$ is a Brownian motion in its own
filtration follows as soon as one observes that $dY^{\ast}_t=\frac{1}{a(t,X^{\ast}_t)}dX^{\ast}_t$, i.e. $Y^{\ast}$ is a local martingale with $[Y^{\ast},Y^{\ast}]_t =t$.

In order to verify the fourth condition, observe that $H^{\ast}(1,x)=F(1,x)=f(x)$. Since by Theorem \ref{mainresult} we have $X^{\ast}_1=Z_1, P^{0,z}$-a.s., the condition holds.

Finally, to demonstrate the martingale property of $H^{\ast}(t,X^{\ast}_t)$ observe that the transition density of $X^{\ast}$ is given by $G(t,x;u,z)$ since in its own filtration
\[
dX^{\ast}_t=a(t,X_t^{\ast})dY^{\ast}_t,
\]
and $Y^{\ast}$ is a Brownian motion so that $X^{\ast}$ satisfies the same SDE (\ref{xii}) as the process $\xi$.
Therefore,
\bean
\bbE^{\bbQ}[f(Z_1)|\cF^{X^{\ast}}_t]&=&\bbE^{\bbQ}[f(X^{\ast}_1)|\cF^{X^{\ast}}_t] \\
&=&\int_{\bbR}f(y) G(t,X^{\ast}_t;1,y)\,dy\\
&=&F(V^{-1}(t), X^{\ast}_t)  \\
&=&H(t,X^{\ast}_t),
\eean
where one to the last equality is due to (\ref{d:f}).
\end{proof}
\begin{remark} \label{r:ia} Note that it follows from Corollary \ref{pdfU} that the conditional density of  $Z_t$ given $\cF^{X^{\ast}}_t$ is $G(t, X^{\ast}_t; V(t), z)$. Note that $G(t,x;u,z)$ converges to the delta function as $t$ converges to $u$. Therefore,  the closer $V(t)$ is  to $t$, the smaller is the uncertainty of the market maker about the value of $Z_t$. Hence, in our case,  as in \cite{BP}, $V(t)-t$ is  a good measurement of the informational advantage of the insider.
\end{remark}
\appendix
\section{Appendix}

 \begin{proposition} \label{hxlim} Suppose  Assumptions \ref{AssZ}, \ref{fs} and \ref{bt} are satisfied. Then, $\frac{h_x}{h}:[0,1] \times \bbR \times [0,1] \times \bbR \mapsto \bbR$ as defined in (\ref{defh}) is (uniformly) bounded.

\end{proposition}
In order to prove the proposition above we need a few preliminary results. The first one is the following classical result due  to  \cite{aronson}.
\begin{lemma} \label{Gambounds} There exist positive constants, $\alpha_1, \alpha_2, M_1,$ and  $M_2$ such that
\[
M_1\, q(\alpha_1 (u-t), x, z) \leq \Gamma(t,x ; u,z) \leq M_2\, q(\alpha_2 (u-t), x, z),
\]
for all $(x,z) \in \bbR^2$ and $u >t$.
\end{lemma}

Next  we need to obtain estimates on the function $h_x/h$. This will be done by following the approach employed in \cite{bc02}. For this purpose define the martingale $L$ by
\[ dL_u= - L_u b(u, \zeta_u)d\beta_u, \qquad u \geq t \] with  $L_t=1$ and let
\[ I(u,z):=\int_0^z b(u,y)dy \qquad N_u:=\int_t^u \left\{I_t(s,
\zeta_s)+ \half b_x(s,\zeta_s)+ \half b^2(s,\zeta_s)\right\}ds. \]
Recall that $\zeta_s = A(s,\xi_s)$ and $d\zeta_s = d \beta_s + b(s,\zeta_s)ds$, where the function $b$ has been defined in  (\ref{bee}).
\begin{remark} \label{justify:bt} Notice that Assumptions \ref{fs} and \ref{bt} ensure that the above formulation make sense.
\end{remark}

 Then,
\[
L^{-1}_u=\exp\left\{I(u,\zeta_u)-I(t,\zeta_t) -N_u\right\}
\]
and a straightforward application of Girsanov's theorem yields
\[
\Gamma(t,x;u,z)=\exp\left(I(u,z)-I(t,x)\right)E^{x, q}_t[\exp(-N_u)|\zeta_u=z]\,q(u-t,x,z),
\]
where $E^{x, q}_t$ is the expectation operator with respect to the
law of the standard Brownian motion starting at $x$ at time $t$. Therefore, (\ref{defh}) becomes
\[h(t,x;u,z)=\exp\left(I(u,z)-I(t,x)\right)E^{x, q}_t[\exp(-N_u)|\zeta_u=z].\]

Observe that $\frac{\partial I(t,x)}{\partial x}=b(t,x)$, which is bounded. Therefore, in order to establish the uniform boundedness of $h_x/h$, we need estimates on
\[
\frac{\frac{\partial}{\partial
x}E^{x, q}_t[\exp(-N_u)|\zeta_u=z]}{E^{x, q}_t[\exp(-N_u)|\zeta_u=z]}.
\]
The next lemma is going to give us an alternative representation of the numerator in the above expression which allows us to obtain a uniform bound on $h_x/h$.

\begin{lemma}\label{derivativeN} Suppose Assumptions \ref{fs} and \ref{bt} are satisfied and  $\zeta_t=x$. Then we have
\[
\frac{\partial N_u}{\partial x}=\int_t^u \left\{b_t(s, \zeta_s)+
\half b_{yy}(s,\zeta_s)+ b(s,\zeta_s)\,b_y(s,\zeta_s)\right\}ds.
\]
Moreover,
\[
\frac{\partial}{\partial
x}E^{x, q}_t[\exp(-N_u)|\zeta_u=z]=-E^{x, q}_t\left[\exp(-N_u)\frac{\partial
N_u}{\partial x}\bigg|\zeta_u=z\right].
\]

\end{lemma}
\begin{proof} In order to prove the first statement note that $\zeta_u=x + W_u$ for some Brownian motion
with $W_t=0$. Since the integrands are
differentiable functions with bounded derivatives, this allows us to
differentiate under the integral sign. Although derivative exists
only almost everywhere, it is no problem since the law of Brownian
motion is absolutely continuous with respect to the Lebesgue
measure. For  the second assertion take an infinitely
differentiable $f:\bbR \mapsto \bbR$ with a compact support.
Therefore, if differentiation inside the expectation is justified, \bean \frac{\partial}{\partial
x}E^{x, q}_t[\exp(-N_u)f(\zeta_u)]&=&E^{x, q}_t\left[\frac{\partial}{\partial
x}\left\{\exp(-N_u)f(\zeta_u)\right\}\right]\\
&=&-E^{x, q}_t\left[\frac{\partial N_u}{\partial
x}\exp(-N_u)f(\zeta_u)\right]+E^{x, q}_t\left[\exp(-N_u)f'(\zeta_u)\right].
\eean
As $\frac{\partial N_u}{\partial x}$, $f$ and $f'$ are bounded, we only need to show $\exp(-N_u)$ is bounded by an integrable function in order to justify the differentiation. Indeed, using Assumptions \ref{fs} and \ref{bt} on the boundedness of function $b$ and its first derivatives, together with the definition of $N_u$, one can easily prove that $N_u \geq K (m_u +x -k) \geq K (m_1 +x -k)$ for some positive  constants $k, K$, where $m_u=\min_{t \leq s \leq u}W_s$. Thus, $\exp(-N_u)$ is bounded above by the random variable $Ce^{-m_1}$ for a positive constant $C$, which may depend on $x$ in a continuous fashion. It follows from the reflection principle for Brownian motion that $-m_1$ has the same law as $|W_1|$. Moreover the random variable $\exp(|W_1|)$ being integrable, we have that $\exp(-N_u)$ is bounded, uniformly in $u$, by an integrable function which does not depend on $x$ when $x$ is restricted to a compact domain. This justifies the differentiation inside the expectation.

 On the other hand, \bean E^{x, q}_t\left[\frac{\partial
}{\partial
x}E^{x, q}_t[\exp(-N_u)|\zeta_u]\,f(\zeta_u)\right]&=&E^{x, q}_t\left[\frac{\partial
}{\partial
x}\left\{E^{x, q}_t[\exp(-N_u)|\zeta_u]f(\zeta_u)\right\}\right]\\
&&-E^{x, q}_t\left[E^{x, q}_t[\exp(-N_u)|\zeta_u]f'(\zeta_u)\right],
\eean thus, we will be done as soon as we have that
\[
\frac{\partial}{\partial
x}E^{x, q}_t[\exp(-N_u)f(\zeta_u)]=E^{x, q}_t\left[\frac{\partial
}{\partial
x}\left\{E^{x, q}_t[\exp(-N_u)|\zeta_u]f(\zeta_u)\right\}\right].
\]
Since $E^{x, q}_t[\exp(-N_u)|\zeta_u]$ is bounded and away from zero whenever $(x,\zeta_u)$
belongs to a bounded domain\footnote{These can be proven by similar arguments that are used in showing  $\exp(-N_u)$ is bounded by an integrable function.} and $f$ has a
compact support, this will follow if $\frac{\partial }{\partial
x}E^{x, q}_t[\exp(-N_u)|\zeta_u]$ is bounded for fixed $u>t$ whenever $(x,\zeta_u)$
belongs to a bounded domain in $\bbR^2$.  To see this note that
\bean (u-t) \frac{\partial}{\partial x}\mbox{log}\Gamma(t,x,u,z)&=&
(u-t)
\frac{\partial}{\partial x}\mbox{log}h(t,x,u,z) + z-x\\
&=&-(u-t)b(t,x) +(u-t)\frac{\frac{\partial }{\partial
x}E^{x, q}_t[\exp(-N_u)|\zeta_u=z]}{E^{x, q}_t[\exp(-N_u)|\zeta_u=z]}+
z-x. \eean The claim follows from the boundedness of $b$, (\ref{est1}) and
Lemma~\ref{Gambounds}. Thus,
\[
\frac{\partial}{\partial
x}E^{x, q}_t[\exp(-N_u)|\zeta_u=z]=-E^{x, q}_t\left[\exp(-N_u)\frac{\partial
N_u}{\partial x}\bigg|\zeta_u=z\right].
\]
\end{proof}

{\sc Proof of Proposition \ref{hxlim}.}\hspace{3mm} First note that in view of Lemma \ref{derivativeN}
\[
\frac{h_x(t,x;u,z)}{h(t,x;u,z)}=- \left(b(t,x)-\frac{E^{x, q}_t[\exp(-N_u)\frac{\partial
}{\partial
x}N_u|\zeta_u=z]}{E^{x, q}_t[\exp(-N_u)|\zeta_u=z]}\right).
\]

As $\frac{\partial}{\partial x} N_u$ is uniformly bounded, in $u$ and $x$, we have
\[
\left|\frac{h_x(t,x;u,z)}{h(t,x;u,z)}\right| \leq |b(t,x)| + \sup_x\left\{\mbox{ ess sup} \left|\frac{\partial}{\partial x} N_u\right|\right\}
\]
by Jensen's inequality. However,  $\sup_{x}\left\{\mbox{ess sup}\left|\frac{\partial}{\partial x} N_u\right|\right\}$ is finite under our assumptions.  Finally, since $b$ is also bounded under our assumptions, the result follows.
\qed
\begin{lemma} \label{j:tech} Suppose that $a(t,z)$ and $\sigma(t)$ satisfy Assumption \ref{a:model}. Let $Z$ satisfy (\ref{SDEsignal}) and $X$ be the process defined in Theorem \ref{t31}. Then,
\begin{enumerate}
\item $b(t,z)=b(t)$ where $b$ is defined by (\ref{bee});
\item the fundamental solution of
\[
w_u(u,z)=\half w_{zz}(u,z) -(b(u,z) w(u,z))_z
\]
is  $\Gamma(t,x;u,z)=q(u-t, x+\int_t^u b(s)ds, z)$ where $q$ is the transition density of standard Brownian motion;
\item For every $z \in \bbR$
\[
E^{0,z}\left[ \int_{0}^{1}H^{2}\left(t,X_t\right)dt\right] <\infty,
\]
where $H(t,x)=F(V^{-1}(t),x)$ with $F$ given by (\ref{d:f}).
\end{enumerate}
\end{lemma}
\begin{proof}
\begin{enumerate}
\item This follows from Remark \ref{rem:AG}.
\item Recall that $\Gamma$ is the transition density of
\[
d\zeta_t=d\beta_t + b(t) dt.
\]
Thus, $\zeta_u-\zeta_t$ has a Gaussian distribution with mean $\int_t^u b(s) ds$ and variance $u-t$.
\item Under the assumptions of the lemma, the process $U$ and $R$ as defined in Lemma \ref{pequivrho} satisfy
\bean
dU_t&=&\sigma(t)d\beta_t +\sigma^2(t)b(t)dt  \\
dR_t&=&dB_t + \left\{ \frac{p_x (t, R_t, U_t)}{p(t,R_t,U_t)} +
b(t)\right\}dt, \eean
where $p(t,x,z)=\Gamma(t,x;V(t),z)$. Therefore,
\[
dR_t=dB_t+ \left\{\frac{U_t-R_t -\int_t^{V(t)}b(s)ds}{V(t)-t} + b(t)\right\}dt.
\]
The solution of the above SDE is given by
\bea
R_t&=&U_0+ \int_0^t b(s) \sigma^2(s) ds - (U_0-R_0)\exp\left(-\int_0^t \frac{1}{V(s)-s}ds\right) \nn \\
&&-\int_0^t \exp\left(-\int_s^t \frac{1}{V(u)-u}du\right) \left(b(s)\sigma^2(s)-b(s)+\frac{\int_s^{V(s)}b(u)du}{V(s)-s}\right)ds \label{eq:rsol}\\
&&+ \int_0^t \sigma(s) \left(1-\exp\left(-\int_s^t \frac{1}{V(u)-u}du\right)\right)d\beta_s + \int_0^t \exp\left(-\int_s^t \frac{1}{V(u)-u}du\right) dB_s. \nn
\eea
Therefore, $R_t$ is a normal variable with bounded (uniformly in $t$) mean and variance. Indeed, due to Remark \ref{rem:lambda_to_0}, $\sup_{s,t}\exp\left(-\int_s^t \frac{1}{V(u)-u}du\right) <\infty$. Moreover, $b$ and $\sigma$ are bounded by assumption.  Therefore, the moment generating function of $R_t$ is a bounded function of time in $[0,1]$.
Next, observe that
\bean
F(t,z)&=& E[f(Z_1)|Z_t=z]=E[f(A^{-1}(1,U_1))|U_t=A(V(t),z)] \\
& \leq & k_1E[\exp(k_2 U_1)|U_t=A(V(t),z)]\\
&=&k_1 \exp\left(k_2\int_t^1 b(s)ds + k_2 A(V(t),z) + \half k_2^2 (1-V(t))\right)\\
&\leq & K \exp(k_2 A(V(t),z))
\eean
due to Assumption \ref{assf} on $f$; the third line is due to the form of the moment generating function of the Gaussian random variable $U_1-U_t$. This in particular implies
\bean
E^{0,z}[H^2(t,X_t)]&=&E^{0,z}[F^2(V^{-1}(t),X_t)]\\
& \leq & K^2 E^{0,z}\left[\exp\left(2 k_2 A(t, X_t)\right)\right]\\
&=&K^2 E^{0,z}\left[\exp\left(2 k_2 R_t \right)\right].
\eean
Note that $sup_{t \in [0,1]} E^{0,z}\left[\exp\left(2 k_2 R_t \right)\right] < \infty$ since the moment generating function of $R$ is bounded. Hence, the claim follows.
\end{enumerate}
\end{proof}

\end{document}